\documentclass[reqno,11pt]{amsart}
\usepackage{amscd,amssymb,verbatim,array}
\usepackage{hyperref}

\usepackage{mathrsfs}
\usepackage{color}

\numberwithin{equation}{section} \theoremstyle{plain}

\newcommand{\Complex}{\mathbb C}
\newcommand{\Real}{\mathbb R}

\newcommand{\ddbar}{\overline\partial}
\newcommand{\pr}{\partial}
\newcommand{\ol}{\overline}
\newcommand{\Td}{\widetilde}
\newcommand{\norm}[1]{\left\Vert#1\right\Vert}
\newcommand{\set}[1]{\left\{#1\right\}}
\newcommand{\To}{\rightarrow}

\newtheorem{theorem}{Theorem}[section]
\newtheorem{lemma}[theorem]{Lemma}
\newtheorem{proposition}[theorem]{Proposition}
\newtheorem{corollary}[theorem]{Corollary}
\newtheorem{definition}[theorem]{Definition}
\newtheorem{ass}[theorem]{Assumption}

\theoremstyle{definition}

\theoremstyle{remark}

\numberwithin{equation}{section}

\newcommand{\abs}[1]{\lvert#1\rvert}

\begin{document}

\title[]
{$G$-invariant Bergman kernel and geometric quantization on complex manifolds with boundary}

\author[]{Chin-Yu Hsiao}
\address{Institute of Mathematics, Academia Sinica and National Center for Theoretical Sciences, 6F, Astronomy-Mathematics Building,
No.1, Sec.4, Roosevelt Road, Taipei 10617, Taiwan}
\email{chsiao@math.sinica.edu.tw or chinyu.hsiao@gmail.com}
\thanks{Chin-Yu Hsiao was partially supported by Taiwan Ministry of Science and Technology projects  108-2115-M-001-012-MY5, 109-2923-M-001-010-MY4 and Academia Sinica Investigator Award.}
\author[]{Rung-Tzung Huang}
\thanks{Rung-Tzung Huang was supported by Taiwan Ministry of Science and Technology project 109-2115-M-008-007-MY2 and 111-2115-M-008 -003 -MY2.}
\address{Department of Mathematics, National Central University, Chung-Li 320, Taiwan}

\email{rthuang@math.ncu.edu.tw}

\author[]{Xiaoshan Li}
\address{School of Mathematics
and Statistics, Wuhan University, Wuhan, 430072, Hubei, China}
\thanks{Xiaoshan Li was supported by National Natural Science Foundation of China (Grant No. 12271411 and 11871380)}
\email{xiaoshanli@whu.edu.cn}

\author[]{Guokuan Shao}
\thanks{Guokuan Shao was supported by National Natural Science Foundation of China (Grant No. 12001549) and the Fundamental Research Funds for the Central Universities, Sun Yat-sen University (Grant No. 22qntd2901)}
\address{School of Mathematics (Zhuhai), Sun Yat-sen University, Zhuhai 519082, Guangdong, China}

\email{shaogk@mail.sysu.edu.cn}


\dedicatory{In memory of Professor Nessim Sibony }

\begin{abstract}
Let $M$ be a complex manifold with boundary $X$, which admits a holomorphic Lie group $G$-action preserving $X$. We establish a full asymptotic expansion for the $G$-invariant Bergman kernel 
under certain assumptions. As an application, we get $G$-invariant version of Fefferman's result about regularity of biholomorphic maps on strongly pseudoconvex domains of $\mathbb C^n$. 
Moreover, we show that the Guillemin-Sternberg map on a complex manifold with boundary is Fredholm by developing reduction to boundary technique, which establish ``quantization commutes with reduction" in this case.
\end{abstract}

\maketitle

\section{Introduction}\label{s-gue190320} 

Let $M$ be a domain in a complex manifold $M'$ with smooth boundary $X$. Let $B$ be the orthogonal projection from $L^2(M)$ 
onto the space of $L^2$ holomorphic functions on $M$ (Bergman projection). The study of boundary behavior of $B$ is a classical 
subject in several complex variables. When $X$ is strongly pseudoconvex, Feffermann~\cite{Fer74} established an asymptotic expansion 
for $B$ at the diagonal. A full asymptotic expansion of $B$ was obtained by Boutete de Monvel and Sj\"ostrand~\cite{BouSj76}. 
The asymptotic of $B$ plays an important role in some important probems in several complex variables. For example, by using 
asymptotic expansion for $B$, Fefferman~\cite{Fer74} established boundary regularity of biholomorphic maps of domains in $\mathbb C^n$. 
When $M$ is a weakly pseudoconvex domain, there are fewer results. In~\cite{HS22}, the first author and Savale obtained a pointwise expansion of $B$ 
at the diagonal in the tangential direction when $M$ is a finite type weakly pseudoconvex domain in $\mathbb C^2$. 
In general, it is very difficult to study the boundary behavior of the Bergman projection on a weakly pseudoconvex domain $M$. Let us see some simple examples and explain our motivation. 

\textbf{(1)} Let 
$M:=\set{(z_1,z_2,z_3)\in\mathbb C^3;\, \abs{z_1}^{4}
+\abs{z_2}^2+\abs{z_3}^2<1}$.
$M$ admits an $S^1$-action: 
\[S^1\times M\to M,\ \ e^{i\theta}\cdot(z_1,z_2,z_3)
=(e^{-i\theta}z_1,e^{i\theta}z_2,e^{i\theta}z_3).\]

\noindent
\textbf{(2)} Let 
\[M:=\Big\{(z_1,z_2,z_3,z_4,z_5,z_6)\in\mathbb C^6 ;\,
(\abs{z_5}^4+\abs{z_6}^2)(\sum^4_{j=1}\abs{z_j}^{2}+z_1z_3
+z_2z_4+\ol z_1\ol z_3+\ol z_2\ol z_4)<1\Big\}.\]
Then, $M$ admits a $G:=S^1\times SU(2)$ action: 
\[\begin{split}
&(e^{i\theta},g)\cdot z=(w_1,w_2,\ldots,w_6),\\
&(w_1,w_2)^t:=g(z_1,z_2)^t,\ \ (w_3,w_4)^t:=\ol g(z_3,z_4)^t,\
\ (w_5,w_6)=(e^{-i\theta}z_5,e^{i\theta}z_6),\\ 
&g\in SU(2),\ \ e^{i\theta}\in S^1,\ \ z\in M,
\end{split}\]
where $z^t$ denotes the transpose of $z$. 

In these examples, all the domains are weakly pseudoconvex but with group action and the boundary reduced spaces of these examples are strongly pseudoconvex CR manifolds (we refer the reader to~\cite[Section 2.5]{HMM} for the details and to Section~\ref{s-gue2204301601} below for the meaning of reduced spaces). In~\cite{HMM}, the first author, Ma and Marinescu showed that the $G$-invariant Szeg\H{o} projection is a complex Fourier integral operator if the reduced space is a strongly pseudoconex CR manifold (the whole CR manifold can be non strongly pseudoconvex). Thus, it is quite natural and interesting to study $G$-invariant Bergman projection on a non strongly pseudoconvex domain with group action. This is the motivation of this work. In this work, we completely study the $G$-invariant Bergman projection on a domain $M$ (can be non strongly pseudoconvex) with group $G$ action under the assumption that the boundary reduced space  with respect to the group $G$ action is non-degenerate. We show that the $G$-invariant Bergman projection on a such domain is a complex Fourier integral operator. As an application, we get $G$-invariant version of Fefferman's result about 
regularity of biholomorphic maps on strongly pseudoconvex domains of $\mathbb C^n$. Since the study of $G$-invariant Bergman 
projection is closely related to geometric quantization, we also study geometric quantization on complex manifolds with boundary. 

We now formulate our main results. We refer to Section~\ref{s-gue2204301052} for some notations and terminology used here. Let $M$ be a relatively compact open subset with smooth boundary $X$ of a complex manifold $M'$ of dimension $n$, $n\geq3$.
Let $\rho\in C^\infty(M',\Real)$ be a defining function of $X$, that is,
\[X=\{x\in M';\, \rho(x)=0\},\ \ M=\{x\in M';\, \rho(x)<0\}\]
and $d\rho(x)\neq0$ at every point $x\in X$.
Then the manifold $X$ is a CR manifold with natural CR structure $T^{1,0}X:=T^{1, 0}M'\cap \mathbb CTX$, where $T^{1,0}M'$ denotes the holomorphic tangent bundle of $M'$.

Suppose that $M'$ admits a compact $d$-dimensional Lie group $G$ action. Let $\mathfrak{g}$ denote the Lie algebra of $G$. 
Let $\mu : M' \to \mathfrak{g}^*$, $\mu_X:=\mu|_X: X\to\mathfrak{g}^*$ be the associated moment maps (cf. Definition~\ref{d-gue170124}). We will work in the following setting.

\begin{ass}\label{a-gue2204300124}
        The $G$-action is holomorphic, preserves the boundary $X$,  
	$0$ is a regular value of $\mu_X$, $G$ acts freely on $\mu^{-1}(0)\cap X$, $\mu^{-1}(0)\cap X\neq\emptyset$, the 
 Levi form is positive or negative near $\mu^{-1}(0)\cap X$.
\end{ass}


Note that the $G$-action is holomorphic means that the $G$-action preserves $J$, where $J$ is the complex structure map on $T^{1,0}M'$. The $G$-action preserves the boundary $X$ means that we can find a defining function $\rho\in C^\infty(M',\Real)$ of $X$ such that $\rho(g \circ x)=\rho(x)$, for every $x\in M'$ and every $g \in G$. 

We take a $G$-invariant Hermitian metric $\langle \cdot \,|\, \cdot \, \rangle$ on $\mathbb{C}TM'$. The $G$-invariant Hermitian metric 
 $\langle \cdot \,|\, \cdot \, \rangle$ on $\mathbb{C}TM'$ induces a 
 $G$-invariant Hermitian metric 
 $\langle \cdot \,|\, \cdot \, \rangle$ on $\oplus_{1\leq p+q\leq n, p, q\in\mathbb N_0}T^{*p,q}M'$, where $T^{*p,q}M'$ 
 denotes the bundle of $(p,q)$ forms on $M'$. 
 From now on, we fix a defining function $\rho\in C^\infty(M',\Real)$ of $X$ such that
\begin{equation}\label{e-gue171006aIm}
\begin{split}
&\mbox{$\langle\,d\rho(x)\,|\,d\rho(x)\,\rangle=1$ on $X$},\\
&\rho(g\circ x)=\rho(x),\ \ \forall x\in M',\ \ \forall g \in G.
\end{split}\end{equation}
Let $(\,\cdot\,|\,\cdot\,)_{M}$ be the $L^2$ inner product on $\Omega^{0,q}_c(M)$ given by
\begin{equation}\label{e-gue171006aIbm}
(\,u\,|\,v\,)_{M}:=\int_M\langle\,u\,|\,v\,\rangle dv_{M'},\ \ u, v\in\Omega^{0,q}_c(M),
\end{equation}
where $dv_{M'}$ is the volume form on $M'$ induced by $\langle\,\cdot\,|\,\cdot\,\rangle$. Let $L^2_{(0,q)}(M)$ be the $L^2$ completion of $\Omega^{0,q}_c(M)$ with respect to
$(\,\cdot\,|\,\cdot\,)_M$. We write $L^2(M):=L^2_{(0,0)}(M)$. 

Let $\ddbar : C^\infty(\overline{M}) \to \Omega^{0, 1}(\overline{M}\,)$ be the Cauchy-Riemann operator on $\ol M$. 
We extend $\ddbar$ to $L^2(M)$:
\[
\ddbar:{\rm Dom\,}\ddbar \subset L^2(M)\rightarrow L^{2}_{(0,1)}(M),
\]
where $u\in{\rm Dom\,}\ddbar$ if we can find $u_j\in C^\infty(\ol M)$, $j=1,2,\ldots$, 
such that $u_j\To u$ in $L^2(M)$ as $j\To+\infty$ and there is a $v\in L^2_{(0,1)}(M)$ such that $\ddbar u_j\To v$ as $j\To+\infty$. We set $\ddbar u:=v$. 
Let 
\begin{equation}\label{e-gue2204301350}
H^0(\overline{M}) : = {\rm Ker\,}\ddbar\subset L^2(M).
\end{equation}
Then $H^0(\overline{M})$ is a (possible infinite dimensional) $G$-representation, its $G$-invariant part is the $G$-invariant $L^2$ holomorphic functions on $\overline{M}$. Let
\begin{equation}\label{e-gue2204301537}
H^0(\overline{M})^G := \left\{ u \in H^0(\overline{M});\,  h^*u =u,~\text{for any}~h \in G \right\}.
\end{equation} 
Let 
\begin{equation}\label{e-gue2204301833}
B_G: L^2(M)\To H^0(\overline{M})^G
\end{equation}
be the orthogonal projection with respect to $(\,\cdot\,|\,\cdot\,)_M$ ($G$-invariant Bergman projection). 
The {\it $G$-invariant Bergman kernel} $B_G(x,y) \in \mathscr{D}'(M\times M)$ is the distribution kernel of $B_G$. 

We introduce some notations. For $x\in X$, let $\mathcal{L}_x$ denote the Levi form of $X$ at $x$ (see \eqref{e-gue230327ycdm}) and let $\det\mathcal{L}_x:=\lambda_1(x)\cdots\lambda_{n-1}(x)$, where $\lambda_j(x)$, $j=1,\ldots,n-1$, are the eigenvalues of $\mathcal{L}_x$ with respect to $\langle\,\cdot\,|\,\cdot\,\rangle$.
For any $\xi \in \mathfrak{g}$, we write $\xi_{M'}$ to denote the vector field on $M'$ induced by $\xi$. That is, 
\begin{equation}\label{e-gue230327ycd}
\mbox{$(\xi_{M'} u)(x)=\frac{\partial}{\partial t}\left(u(\exp(t\xi)\circ x)\right)|_{t=0}$, for any $u\in C^\infty(M')$}. 
\end{equation} 
For $x \in M'$, set
\begin{equation}\label{e-gue2204290306}
\underline{\mathfrak{g}}_x={\rm Span\,}\left\{ \xi_{M'}(x);\, \xi\in\mathfrak{g}\, \right\}.
\end{equation} 
Fix $x\in\mu^{-1}(0) \cap X$, consider the linear map 
\[
\renewcommand{\arraystretch}{1.2}
\begin{array}{rll}
R_x:\underline{\mathfrak{g}}_x&\To&\underline{\mathfrak{g}}_x,\\
u&\To& R_xu,\ \ \langle\,R_xu\,|\,v\,\rangle=\langle\,d\omega_0(x)\,,\,Ju\wedge v\,\rangle,
\end{array}
\]
where $\omega_0(x)=J^t(d\rho)(x)$, $J^t$ is the complex structure map on $T^*M'$. 
Let $\det R_x=\mu_1(x)\cdots\mu_d(x)$, where $\mu_j(x)$, $j=1,2,\ldots,d$, are the eigenvalues of $R_x$. 
Fix $x\in\mu^{-1}(0)\cap X$, put $Y_x=\set{g\circ x;\, g\in G}$. $Y_x$ is a $d$-dimensional submanifold of $X$. The $G$-invariant Hermitian metric $\langle\,\cdot\,|\,\cdot\,\rangle$ induces a volume form $dv_{Y_x}$ on $Y_x$. Put 
\[
V_{{\rm eff\,}}(x):=\int_{Y_x}dv_{Y_x}.
\]

The first main result of this work is the following 

\begin{theorem}\label{t-510}
With the notations and assumptions above and recall that we work with Assumption~\ref{a-gue2204300124}. Let $\tau\in C^\infty(\ol M)$ with ${\rm supp\,}\tau\cap\mu^{-1}(0)\cap X=\emptyset$. 
Then, $\tau B_G\equiv0\mod C^\infty(\ol M\times\ol M)$, $B_G\tau\equiv0\mod C^\infty(\ol M\times\ol M)$. 

Let $p\in\mu^{-1}(0)\cap X$. Let $U$ be an open local coordinate patch of $p$ in $M'$, $D:=U\cap X$. 
If Levi form is negative on $D$, then 
\begin{equation}\label{e-gue230411yyd}
B_G(z,w)\equiv0\mod C^\infty((U\times U)\cap(\ol M\times\ol M)).
\end{equation}
Suppose that the Levi form is positive on $D$. Then, 
\begin{equation}\label{e-gue230319ycdem}
B_G(z,w)\equiv\int^{+\infty}_0e^{it\Psi(z,w)}b(z,w,t)dt\mod C^\infty((U\times U)\cap(\ol M\times\ol M)),
\end{equation}
where 
\begin{equation}\label{e-gue230319ycdfm}
\begin{split}
&b(z,w,t)\in S^{n-\frac{d}{2}}_{1,0}(((U\times U)\cap(\ol M\times\ol M))\times\mathbb R_+),\\
&b(z,w,t)\sim\sum^{+\infty}_{j=0}t^{n-\frac{d}{2}-j}b_j(z,w)\ \ \mbox{in $S^{n-\frac{d}{2}}_{1,0}(((U\times U)\cap(\ol M\times\ol M))\times\mathbb R_+)$},\\
&b_j(z,w)\in C^\infty((U\times U)\cap(\ol M\times\ol M)),\ \ j=0,1,2,\ldots,
\end{split}
\end{equation}
\begin{equation}\label{e-gue230327ycda}
b_0(z,z)=2^{d}\frac{1}{V_{{\rm eff\,}}(x)}\pi^{-n+\frac{d}{2}}\abs{\det R_x}^{-\frac{1}{2}}\abs{\det\mathcal{L}_{x}},\ \ \forall x\in\mu^{-1}(0)\cap D,
\end{equation}
and 
\begin{equation}\label{e-gue230319ycdgm}
\begin{split}
&\Psi(z,w)\in C^\infty(((U\times U)\cap(\ol M\times\ol M))),\ \ {\rm Im\,}\Psi\geq0,\\
&\Psi(z,z)=0, \ z\in\mu^{-1}(0)\cap D,\\
&\mbox{${\rm Im\,}\Psi(z,w)>0$ if $(z,w)\notin{\rm diag\,}((\mu^{-1}(0)\cap D)\times(\mu^{-1}(0)\cap D))$},\\
&d_x\Psi(x,x)=-\omega_0(x)-id\rho(x),\ \ 
d_y\Psi(x,x)=\omega_0(x)-id\rho(x),\ \ x\in\mu^{-1}(0)\cap D,\\
&\mbox{$\Psi|_{D\times D}=\Phi$, $\Phi$ is the phase as in~\cite[Theorem 1.5]{HH}}.
\end{split}
\end{equation}  

Moreover, let $z=(x_1,\ldots,x_{2n-1},\rho)$ be local coordinates of $M'$ defined near $p$ in $M'$ with $x(p)=0$ and $x=(x_1,\ldots,x_{2n-1})$ are local coordinates of $X$ defined near $p$ in $X$. Then, 
\begin{equation}\label{e-gue230319ycdaIm}
\mbox{$\Psi(z,w)=\Phi(x,y)-i\rho(z)(1+f(z))-i\rho(w)(1+\ol f(w))+O(\abs{(z,w)}^3)$ near $(p,p)$},
\end{equation}
where $f\in C^\infty$, $f=O(\abs{z})$. 
\end{theorem} 

The above theorem lays a foundation to the study of Toeplitz quantization on complex manifolds with boundary.
We refer the reader to the discussion before \eqref{e-gue171006aIb} for the meaning of $F\equiv G \mod C^\infty((U\times U)\cap(\ol M\times\ol M))$. 

Before we formulate our main result about geometric quantization on complex manifolds with boundary, we give some historic remark about geometric quantization theory. 
The famous geometric quantization conjecture of Guillemin and Sternberg~\cite{GS} states that for a compact pre-quantizable symplectic manifold admitting a Hamiltonian action of a compact Lie group, the principle of ``quantization commutes with reduction" holds. This conjecture was first proved independently by Meinrenken \cite{M96} and Vergne \cite{V96} for the case where the Lie group is abelian, and by Meinrenken \cite{M98} in the general case, then Tian-Zhang \cite{TZ98} gave a purely analytic proof in general case with various generalizations, see \cite{V02} for a survey and complete references on this subject. In the case of a non-compact symplectic manifold $M$ which has a compact Lie group action $G$, this question was solved by Ma-Zhang \cite{MZ09, MZ14} as a solution to a conjecture of Vergne in her ICM 2006 plenary lecture \cite{V07}, see \cite{M10} for a survey. Paradan \cite{P11} gave a new proof, cf. also the recent work \cite{HS17}. A natural choice for the quantum spaces of a compact symplectic manifold is the kernel of the Dirac operator. 

In~\cite{MZ}, Ma-Zhang established the asymptotic expansion 
of the $G$-invariant Bergman kernel for a positive line bundle $L$ over 
a compact symplectic manifold $M$ and by using the asymptotic expansion 
of $G$-invariant Bergman kernel, they could establish the
``quantization commutes with reduction'' theorem
when the power of the line bundle $L$ is high enough. In~\cite{HH}, the first and second authors established the asymptotic expansion 
of the $G$-invariant Szeg\H{o} kernel for $(0,q)$ forms on a non-degenerate CR manifold and 
they could establish the
``quantization commutes with reduction'' theorem
when the CR manifold admits a circle action. 

The quantization of strongly pseudoconvex or more generally contact manifolds via the Szeg\H{o} projector or its generalizations was developed by Boutet de Monvel and Guillemin \cite{BG81} and can be applied to the K\"{a}hler quantization by using the above construction (see e. g. \cite{Ch06, Pa03}). In \cite{HMM}, the first author, Ma and Marinescu study the quantization of CR manifolds and the principle of ``quantization commutes with reduction". An important difference between the CR setting and the K\"{a}hler/symplectic setting is that the quantum spaces in the case of compact K\"{a}hler/symplectic manifolds are finite dimensional, whereas for the compact strongly pseudoconvex CR manifolds that they consider the quantum spaces consisting of CR functions are infinite dimensional.

For manifolds with boundary, in \cite{TZ99}, Tian-Zhang extended the results in \cite{TZ98} to the case where the compact symplectic manifold with a non-empty boundary under the assumption that the preimage of the moment map of the regular value $0$ in the dual of the Lie algebra does not touch the boundary. The quantum spaces considered in \cite{TZ99} are the kernel of the Dirac operator with Atiyah-Patodi-Singer type boundary conditions \cite{APS75} and hence finite dimensional. Following the same line of \cite{HMM}, in this paper we study the quantization of complex manifolds with boundary and the principle of ``quantization commutes with reduction". The quantum spaces we consider are the spaces of $L^2$ holomorphic functions and could be infinite dimensional. 

We now formulate our main results. 
By Assumption~\ref{a-gue2204300124}, $\mu^{-1}_X(0)$ is a $d$-codimensional submanifold of $X$. We decompose $\mu^{-1}(0)\cap X$ into two parts $\widehat X$ and $\widetilde X$ on which the Levi-form is strongly pseudoconvex and pseudoconcave, respectively. From now on, we assume that $\widehat X$ is non-emptey. Let
\begin{equation}\label{e-gue230411yyda}
\widehat X_G:=\widehat X/G,~\widetilde X_G=\widetilde X/G.
\end{equation} 
It was proved in~\cite[Theorem 2.6]{HMM} that $\widehat X_G$ is a compact CR manifold. Let 
\[\ddbar_b: {\rm Dom\,}\ddbar_b\subset L^2(\widehat X_G)\To L^2_{(0,1)}(\widehat X_G)\]
be the tangential Cauchy-Riemann operator. For every $s \in \mathbb{R}$, let $W^s(\overline{M})$ and $W^s(\widehat X_G)$ denotes the Sobolev spaces of $\overline{M}$ and $\widehat X_G$ of order $s$ (see the discussion after Definition~\ref{def-111217}, for the precise meaning of $W^s(\overline{M})$).  Let $(\,\cdot\,|\,\cdot\,)_{\widehat X_G}$ be the $L^2$ inner product on $L^2(\widehat X_G)$ induced naturally by $\langle\,\cdot\,|\,\cdot\,\rangle$. For every $s\in \mathbb{R}$, put
 \begin{equation}\label{e-gue240411yydb}
\begin{split}
& H^0(\overline{M})_s := \left\{ u \in W^s(\overline{M});\, \ddbar u =0~\text{in the sense of distributions} \right\}, \\
& H^0_b(\widehat X_G)_{s} := \left\{ u \in W^s(\widehat X_G);\, \ddbar_bu =0~\text{in the sense of distributions} \right\}, \\
& H^0(\overline{M})^G_s := \left\{ u \in H^0(\overline{M})_s;\, h^*u=u~\text{in the sense of distributions for all $h \in G$} \right\}.
\end{split}
\end{equation}
We write $H^0_b(\widehat X_G):=H^0_{b}(\widehat X_G)_0$.
Let $\iota_{\widehat X} : \widehat X\hookrightarrow X$ be the natural inclusion and let $\iota^*_{\widehat X} : C^\infty(X) \to C^\infty(\widehat X)$ be the pull-back by $\iota_{\widehat X}$. Let $\iota_{G, \widehat X}: C^\infty(\widehat X)^G \to C^\infty(\widehat X_G)$ be the natural identification. Let
\begin{equation}\label{e-gue2204301555m}
\tilde\sigma_{G} : H^0(\overline{M})^G \cap C^\infty(\ol M) \to H^0_b(\widehat X_G), \qquad \tilde\sigma_G =\iota_{G, \widehat X} \circ \iota^*_{\widehat X}\circ\gamma,
\end{equation}
where $\gamma$ denotes the operator of the restriction to the boundary $X$.
The map (\ref{e-gue2204301555m}) is well defined. The map $\tilde\sigma_G$ does not extend to a bounded operator on $L^2$, so it is necessary to consider its extension to Sobolev spaces. 
We can check that $\tilde\sigma_G$ extends by density to a bounded operator
\begin{equation}\label{e-gue2204301610m}
\tilde\sigma_G =\tilde\sigma_{G, s} : H^0(\overline{M})^G_s \to H^0_b(\widehat X_G)_{s-\frac{d}{4}-\frac{1}{2}},~\text{for every $s \in \mathbb{R}$}
\end{equation}
(see Theorem~\ref{t-701} and Theorem~\ref{t-gue230321yyd} below). For every $s\in\mathbb R$, put 
\begin{equation}\label{e-gue230328yydaz}
{\rm Coker\,}\tilde\sigma_{G,s}={\rm Coker\,}\tilde\sigma_G:=\{u\in H^0_b(\widehat X_G)_{s-\frac{d}{4}-\frac{1}{2}};\, (\,u\,|\,\tilde\sigma_{G,s}v)_{\widehat X_G}=0, \forall v\in H^0(\overline M)^G_s\cap C^\infty(\overline M)\}.
\end{equation} 
The following is our second main result

\begin{theorem}\label{t-gue230411yyd}
Let $M$ be a relatively compact open subset with smooth boundary $X$ of a complex manifold $M'$ of dimension $n$, $n\geq3$. Let $G$ be a compact Lie group acting on $M'$ such that Assumption~\ref{a-gue2204300124} holds.  With the notations used above, assume that ${\rm dim\,}_{\mathbb R}\widehat X_G\geq 5$. 
Then, for every $s \in \mathbb{R}$, the Guillemin-Sternberg map \eqref{e-gue2204301610m} is Fredholm. More precisely, $\operatorname{Ker}\tilde\sigma_{G, s}$ and ${\rm Coker\,}\tilde\sigma_{G,s}$ are finite dimensional subspaces of $H^0(\overline{M})^G\cap C^\infty(\ol M)^G$ and $H^0_b(\widehat X_G)\cap C^\infty(\widehat X_G)$ respectively, $\operatorname{Ker}\tilde\sigma_{G, s}$ and ${\rm Coker\,}\tilde\sigma_{G,s}$ are independent of $s$.    
\end{theorem}

The assumption ${\rm dim\,}_{\mathbb R}\widehat X_G\geq 5$ in Theorem~\ref{t-gue230411yyd} can be replaced by $\ddbar_{b, \widehat X_G}$ has closed range, where $\ddbar_{b, \widehat X_G}$ denotes the tangential Cauchy-Riemann operator on $\widehat X_G$. 

Theorem~\ref{t-gue230411yyd} tells us that up to some finite dimensional spaces, the quantum space $H^0(\overline{M})^G$ is isomorphic to the space of $L^2$ CR functions on $\widehat X_G$. 

Suppose that 
\begin{equation}\label{e-gue230411yydc}
\begin{split}
\mbox{$0$ is a regular value of $\mu$, $G$ acts freely on $\mu^{-1}(0)$.}
\end{split}
\end{equation}
Under \eqref{e-gue230411yydc}, $\mu^{-1}(0)$ is a $d$-codimensional submanifold of $M'$. Let
\begin{equation}\label{e-gue22-4301336}
M_G' := \mu^{-1}(0)/G, \quad M_G:=(\mu^{-1}(0) \cap M)/G.
\end{equation}
In Theorem~\ref{t-gue230328yyd} below, we will show that $M_G$ is a complex manifold in $M'_G$ with smooth boundary $X_G$. In fact, $X_G=\widehat X_G\cup \widetilde X_G$ and thus the the boundary $X_G$ is non-degenerate, hence $M_G$ is a domain in the complex manifold $M'_G$ with non-degenerate boundary. 

Let $\iota : \mu^{-1}(0)\cap\ol M\hookrightarrow \overline{M}$ be the natural inclusion and let $\iota^* : C^\infty(\overline{M}) \to C^\infty(\mu^{-1}(0)\cap\ol M)$ be the pull-back by $\iota$. Let $\iota_G: C^\infty(\mu^{-1}(0)\cap\ol M)^G \to C^\infty(\overline{M}_G)$ be the natural identification. Let
\begin{equation}\label{e-gue2204301555}
\sigma_G : H^0(\overline{M})^G \cap C^\infty(\overline{M}) \to H^0(\overline{M}_G), \qquad \sigma_G =\iota_G \circ \iota^*.
\end{equation}
The map (\ref{e-gue2204301555}) is well defined, see the construction of the complex reduction in Section~\ref{s-gue2204301601}. The map $\sigma_G$ does not extend to a bounded operator on $L^2$, so it is necessary to consider its extension to Sobolev spaces. 
Actually, we have 
\[\sigma_G=P_{M_G}\sigma_1(P^*P)^{-1}P^*\ \ \mbox{on $H^0(\overline{M})^G \cap C^\infty(\overline{M})$},\]
where $P_{M_G}$ and $P$ are Poisson operators on $M$ and $M_G$ respectively and $\sigma_1$ is the CR Guillemin-Sternberg map introduced in~\cite[(1.5)]{HMM}. From~\cite[Theorem 5.3]{HMM} and the regularity property for Poisson operator (see Section~\ref{s-gue171010}), we can check that $\sigma_G$ extends by density to a bounded operator
\begin{equation}\label{e-gue2204301610}
\sigma_G = \sigma_{G, s} : H^0(\overline{M})^G_s \to H^0(\overline{M}_G)_{s-\frac{d}{4}},~\text{for every $s \in \mathbb{R}$}.
\end{equation}
This operator can be thought as a Guillemin-Sternberg map in the setting of complex manifolds with boundary. It maps the ``first quantize and then reduce" space (the space of $G$-invariant Sobolev holomorphic functions on $\overline{M}$) to the ``first reduce and then quantize" space (the space of Sobolev holomorphic functions on $\overline{M}_G$). Indeed, from the point of view of quantum mechanics, the Hilbert space structures play an essential role. It is natural, then, to investigate the extent to which the holomorphic Guillemin-Sternberg map is Fredholm. Let $(\,\cdot\,|\,\cdot\,)_{M_G}$ be the $L^2$ inner product on $L^2(M_G)$ induced naturally by $\langle\,\cdot\,|\,\cdot\,\rangle$. For every $s\in\mathbb R$, put 
\begin{equation}\label{e-gue230328yyda}
{\rm Coker\,}\sigma_{G,s}={\rm Coker\,}\sigma_G:=\{u\in H^0(\ol M_G)_{s-\frac{d}{4}};\, (\,u\,|\,\sigma_G v)_{M_G}=0, \forall v\in H^0(\overline M)^G_s\cap C^\infty(\overline M)\}.
\end{equation} 

The third main result of this work is the following.

\begin{theorem}\label{t-que2204301824}
Let $M$ be a relatively compact open subset with smooth boundary $X$ of a complex manifold $M'$ of dimension $n$, $n\geq3$. Let $G$ be a compact Lie group acting on $M'$ such that Assumption~\ref{a-gue2204300124} and \eqref{e-gue230411yydc} hold. With the notations used above, assume that ${\rm dim\,}_{\mathbb C}M_G\geq 3$.
Then, for every $s \in \mathbb{R}$, the holomorphic Guillemin-Sternberg map \eqref{e-gue2204301610} is Fredholm. More precisely, $\operatorname{Ker} \sigma_{G, s}$ and ${\rm Coker\,}\sigma_{G,s}$ are finite dimensional subspaces of $H^0(\overline{M})^G\cap C^\infty(\ol M)^G$ and $H^0(\overline{M}_G)\cap C^\infty(M_G)$ respectively, $\operatorname{Ker} \sigma_{G, s}$ and ${\rm Coker\,}\sigma_{G,s}$ are independent of $s$.
\end{theorem}

In should be mentioned that the condition ${\rm dim\,}_{\mathbb C}M_G\geq 3$ in Theorem~\ref{t-que2204301824} can be replaced by $\ddbar_{b,X_G}$ has closed range.

Until further notice, we will not assume \eqref{e-gue230411yydc}. 

Suppose that $M'$ admits another compact holomorphic Lie group action $H$ such that $H$ commutes with $G$ and $H$ preserves the boundary $X$. Recall that $\mu^{-1}(0)\cap X=\widehat X\cup\widetilde X$ on which the Levi form is strongly pseudoconvex and pseudoconcave, respectively. Let 
\[\mathcal{R}=\set{\mathcal{R}_m;\, m=1,2,\ldots}\]
denote the set of all irreducible unitary representations of the group $H$, including only one representation from each equivalence class.
For each $\mathcal{R}_m$, we  write $\mathcal{R}_m$ as a matrix $\left(\mathcal{R}_{m,j,k}\right)^{d_m}_{j,k=1}$, where $d_m$ is the dimension of $\mathcal{R}_m$. Fix a Haar measure $d\nu(h)$ on $H$ so that $\int_Hd\nu(h)=1$. Take an irreducible unitary representation $\mathcal{R}_m$, for every $h\in H$, put 
\[\chi_m(h):={\rm Tr\,}\left(\mathcal{R}_{m,j,k}(h)\right)^{d_m}_{j,k=1}=\sum^{d_m}_{j=1}\mathcal{R}_{m,j,j}(h).\]
Let $u\in C^\infty(M')$ be a smooth function. The $m$-th Fourier component of $u$ is given by 
\[u_m(x):=d_m\int_H (h^*u)(x)\ol{\chi_m(h)}d\nu(h)\in C^\infty(M').\]
For every $m\in\mathbb N$, put 
\[C^\infty_m(M'):=\set{f\in C^\infty(M');\, \mbox{there is an $F\in C^\infty(M')$ such that $f=F_m$ on $M'$}}.\]
For every $m\in\mathbb N$, we define $C^\infty_m(\ol M)$, $C^\infty_m(\widehat X_G)$ in the standard way. 
For every $m\in\mathbb N$, let 
\begin{equation}\label{e-gue230328ycdm}
\begin{split}
&H^0(\ol M)_{(m)}:=H^0(\ol M)\cap C^\infty_m(\ol M),\\
&H^0(\ol M)^G_{(m)}:=H^0(\ol M)^G\cap C^\infty_m(\ol M),\\
&H^0_b(\widehat X_G)_{(m)}:=H^0_b(\widehat X_G)\cap C^\infty_m(\widehat X_G).
\end{split}
\end{equation}
Let $\mathfrak{h}$ denote the Lie algebra of $H$. 
For any $\xi \in \mathfrak{h}$, as \eqref{e-gue230327ycd}, we write $\xi_{M',H}$ to denote the vector field on $M'$ induced by $\xi$. 
For $x \in M'$, set
\begin{equation}\label{e-gue2204290306m}
\underline{\mathfrak{h}}_{x}={\rm Span\,}\left\{ \xi_{M',H}(x);\, \xi\in\mathfrak{h}\, \right\}.
\end{equation} 
We assume that 
\begin{equation}\label{e-gue230328ycdp}
T^{1,0}_x\widehat X\oplus T^{0,1}_x\widehat X\oplus \underline{\mathfrak{h}}_{x}=\mathbb CT_x\widehat X,\ \ \mbox{for every $x\in \widehat X$}, 
\end{equation}
where $T^{1,0}_x\widehat X:=T^{1,0}_xM'\cap\mathbb CT_x\widehat X$, $T^{0,1}_x\widehat X:=T^{0,1}_xM'\cap\mathbb CT_x\widehat X$, $T^{1,0}M'$ and $T^{0,1}M'$ denotes the holomorphic tangent bundle of $M'$ and the anti-holomorphic tangent bundle of $M'$ respectively. We can repeat the proof of~\cite[Theorem 3.1, Appendix]{HHLS20} with minor change and deduce that 
\begin{equation}\label{e-gue230328ycdq}
\begin{split}
&\mbox{${\rm dim\,}H^0(\ol M)_{(m)}<+\infty$, ${\rm dim\,}H^0(\ol M)^G_{(m)}<+\infty$, ${\rm dim\,}H^0_b(\widehat X_G)_{(m)}<+\infty$, for every $m\in\mathbb N$},\\
&H^0(\ol M)=\oplus_{m\in\mathbb N}H^0(\ol M)_{(m)},\ \ H^0(\ol M)^G=\oplus_{m\in\mathbb N}H^0(\ol M)^G_{(m)},\ \
H^0_b(\widehat X_G)=\oplus_{m\in\mathbb N}H^0_b(\widehat X_G)_{(m)}.
\end{split}
\end{equation}

From Theorem~\ref{t-gue230411yyd}, Theorem~\ref{t-que2204301824} and \eqref{e-gue230328ycdq}, we deduce 

\begin{theorem}\label{t-gue230328yydr}
With the same assumptions used in Theorem~\ref{t-gue230411yyd}, suppose that $M'$ admits another compact holomorphic Lie group action $H$ such that $H$ commutes with $G$ and $H$ preserves the boundary $X$. Under the same notations above and assume that \eqref{e-gue230328ycdp} holds. Then, for $\abs{m}\gg1$, we have 
\[{\rm dim\,}H^0(\ol M)^G_{(m)}={\rm dim\,}H^0_b(\widehat X_G)_{(m)}.\] 

Assume further that \eqref{e-gue230411yydc} holds. Then, for $\abs{m}\gg1$, we have 
\[{\rm dim\,}H^0(\ol M)^G_{(m)}={\rm dim\,}H^0(\ol M_G)_{(m)}.\] 
\end{theorem}
 As an application of Theorem~\ref{t-510}, we establish $G$-invariant version of Fefferman's result about regularity of biholomorphic maps. Let $M_1$, $M_2$ be bounded domains in $\mathbb C^n$. Assume that $M_j, j=1, 2$ admit a compact holomorphic Lie group action $G$. Let $F: M_1\rightarrow M_2$ be a holomorphic map. $F$ is said to be $G$-invariant if $F(g\circ z)=F(z)$, for all $z\in M_1$ and $g\in G$. Then from Theorem~\ref{t-510} and by using the argument in~\cite{BL80}, we have (see Section~\ref{s-gue230324yyd}, for the details)

 \begin{theorem}\label{t-gue230329yyd}
 Let $M_1$, $M_2$ be bounded domains in $\mathbb C^n$ with smooth boundary, $n\geq3$. Assume that $M_j$ admits a compact holomorphic Lie group action $G$ and Assumption~\ref{a-gue2204300124} holds, for each $j=1, 2$. 
Let $F: M_1\To M_2$ be a $G$-invariant holomorphic map. Assume that the induced map of $F$ on the quotient space still denoted by $F: M_1/G\rightarrow M_2/G$ is onto, one-to-one and the differential of $F$ is invertible everywhere on the regular part of $M_1/G$.
 Then, $F$ extends smoothly to the boundary.
 \end{theorem}

In the end of this section, we give a simple example. 
Let 
\[M:=\set{(z_1,z_2,z_3,z_4)\in\mathbb C^4;\, \abs{z_1}^{4}
+\sum^4_{j=2}\abs{z_j}^2<1}.\]
$M$ admits a $S^1$-action: 
\[S^1\times M\to M,\ \ e^{i\theta}\cdot(z_1,\ldots,z_4)
=(e^{-i\theta}z_1,e^{i\theta}z_2,\ldots,e^{i\theta}z_4).\]
We can show that Assumption~\ref{a-gue2204300124} holds in this example (see~\cite[Section 2.5]{HMM} for the details). Moreover, it is straightforward to check that $0\in\mathbb C^4$ is a critical point of $\mu$ and hence \eqref{e-gue230411yydc} does not hold.  
In this example, we have Theorem~\ref{t-510} and Theorem~\ref{t-gue230411yyd}. Since $M_G$ has singularities, we do not know if we have Theorem~\ref{t-gue230328yydr}. It is quite interesting to see if we have Theorem~\ref{t-gue230328yydr} for singular reduction. 

Let us consider the shell domain
\[M:=\set{(z_1,z_2,z_3,z_4)\in\mathbb C^4;\, \frac{1}{2}<\abs{z_1}^{4}
+\sum^4_{j=2}\abs{z_j}^2<1}.\]
Then, Assumption~\ref{a-gue2204300124} and \eqref{e-gue230411yydc} hold in this example and we have Theorem~\ref{t-510}, Theorem~\ref{t-gue230411yyd} and Theorem~\ref{t-gue230328yydr} for this example. 

\section{Preliminaries}\label{s-gue2204301052}

\subsection{Some standard notations}\label{s-gue230316yyd}
We use the following notations: $\mathbb N=\set{1,2,\ldots}$,
$\mathbb N_0=\mathbb N\cup\set{0}$, $\Real$
is the set of real numbers, $\ol\Real_+:=\set{x\in\Real;\, x\geq0}$.
For a multiindex $\alpha=(\alpha_1,\ldots,\alpha_m)\in\mathbb N_0^m$,
we set $\abs{\alpha}=\alpha_1+\cdots+\alpha_m$. For $x=(x_1,\ldots,x_m)\in\Real^m$ we write
\[
\begin{split}
&x^\alpha=x_1^{\alpha_1}\ldots x^{\alpha_m}_m,\quad
 \pr_{x_j}=\frac{\pr}{\pr x_j}\,,\quad
\pr^\alpha_x=\pr^{\alpha_1}_{x_1}\ldots\pr^{\alpha_m}_{x_m}=\frac{\pr^{\abs{\alpha}}}{\pr x^\alpha}\,,\\
&D_{x_j}=\frac{1}{i}\pr_{x_j}\,,\quad D^\alpha_x=D^{\alpha_1}_{x_1}\ldots D^{\alpha_m}_{x_m}\,,
\quad D_x=\frac{1}{i}\pr_x\,.
\end{split}
\]
Let $z=(z_1,\ldots,z_m)$, $z_j=x_{2j-1}+ix_{2j}$, $j=1,\ldots,m$, be coordinates of $\Complex^m$,
where
$x=(x_1,\ldots,x_{2m})\in\Real^{2m}$ are coordinates in $\Real^{2m}$.
We write
\[
\begin{split}
&z^\alpha=z_1^{\alpha_1}\ldots z^{\alpha_m}_m\,,\quad\ol z^\alpha=\ol z_1^{\alpha_1}\ldots\ol z^{\alpha_m}_m\,,\\
&\pr_{z_j}=\frac{\pr}{\pr z_j}=
\frac{1}{2}\Big(\frac{\pr}{\pr x_{2j-1}}-i\frac{\pr}{\pr x_{2j}}\Big)\,,\quad\pr_{\ol z_j}=
\frac{\pr}{\pr\ol z_j}=\frac{1}{2}\Big(\frac{\pr}{\pr x_{2j-1}}+i\frac{\pr}{\pr x_{2j}}\Big),\\
&\pr^\alpha_z=\pr^{\alpha_1}_{z_1}\ldots\pr^{\alpha_m}_{z_m}=\frac{\pr^{\abs{\alpha}}}{\pr z^\alpha}\,,\quad
\pr^\alpha_{\ol z}=\pr^{\alpha_1}_{\ol z_1}\ldots\pr^{\alpha_m}_{\ol z_m}=
\frac{\pr^{\abs{\alpha}}}{\pr\ol z^\alpha}\,.
\end{split}
\]

Let $\Omega$ be a $C^\infty$ orientable paracompact manifold.
We let $T\Omega$ and $T^*\Omega$ denotes the tangent bundle of $\Omega$ and the cotangent bundle of $\Omega$ respectively.
The complexified tangent bundle of $\Omega$ and the complexified cotangent bundle of $\Omega$
will be denoted by $\Complex T\Omega$ and $\Complex T^*\Omega$ respectively. We write $\langle\,\cdot\,,\cdot\,\rangle$
to denote the pointwise duality between $T\Omega$ and $T^*\Omega$.
We extend $\langle\,\cdot\,,\cdot\,\rangle$ bilinearly to $\Complex T\Omega\times\Complex T^*\Omega$.

Let $E$ be a $C^\infty$ complex vector bundle over $\Omega$. The fiber of $E$ at $x\in\Omega$ will be denoted by $E_x$.
Let $F$ be another vector bundle over $\Omega$. We write
$F\boxtimes E^*$ to denote the vector bundle over $\Omega\times\Omega$ with fiber over $(x, y)\in\Omega\times\Omega$
consisting of the linear maps from $E_y$ to $F_x$.

Let $Y\subset \Omega$ be an open set. The spaces of smooth sections of $E$ over $Y$ and distribution sections of $E$ over $Y$ will be denoted by $C^\infty(Y, E)$ and $\mathscr D'(Y, E)$ respectively.
Let $\mathscr E'(Y, E)$ be the subspace of $\mathscr D'(Y, E)$ whose elements have compact support in $Y$ and set 
$C^\infty_c(Y,E):=C^\infty(Y,E)\bigcap\mathscr E'(Y, E)$. Fix a volume form on $Y$ and a Hermitian metric on $E$, we get a 
natural $L^2$ inner product $(\,\cdot\,|\,\cdot\,)$ on $C^\infty_c(Y,E)$. Let $L^2(Y,E)$ be the completion of $C^\infty_c(Y,E)$ with respect 
to $(\,\cdot\,|\,\cdot\,)$ and the $L^2$ inner product $(\,\cdot\,|\,\cdot\,)$ can be extended to $L^2(Y,E)$ by density. Let $\norm{\cdot}$ be the $L^2$ norm corresponding to the $L^2$ inner product $(\,\cdot\,|\,\cdot\,)$. 
For every $s\in\mathbb R$, let $L_s: \mathscr D'(Y,E)\To\mathscr D'(Y,E)$ be a properly supported classical 
elliptic pseudodifferential operator of order $s$ on $Y$ with values in $E$. Define 
\[W^s(Y,E):=\set{u\in \mathscr D'(Y,E);\, L_s u\in L^2(Y,E)}\]
and for $u\in W^s(Y,E)$, let $\norm{u}_s:=\norm{L_s u}$. 
We call $W^s(Y, E)$ the Sobolev space of order $s$ of sections of $E$ over $Y$ (with respect to $L_s$) and for $u\in W^s(Y,E)$, we call the number $\norm{u}_s$ the Sobolev norm of $u$ of order $s$ (with respect to $L_s$).
Put
\begin{gather*}
W^s_{\rm loc\,}(Y, E)=\big\{u\in\mathscr D'(Y, E);\, \varphi u\in W^s(Y, E),
     \,\forall\varphi\in C^\infty_c(Y)\big\}\,,\\
      W^s_{\rm comp\,}(Y, E)=W^s_{\rm loc}(Y, E)\cap\mathscr E'(Y, E)\,.
\end{gather*}

Let $U, V$ be open sets of $\Omega$. 
Let $F: C^\infty_c(V)\to\mathscr{D}'(U)$ be a continuous operator and let $F(x,y)\in\mathscr{D}'(U\times V)$ be the distribution kernel of $F$. In this work, we will identify $F$ with $F(x,y)$.
We say that
$F$ is a smoothing operator 
if $F(x,y)\in C^\infty(U\times V)$. 
Note that the following conditions are equivalent.
\begin{equation}
    \begin{split}
        &F(x,y)\in C^\infty(U\times V).\\\
        &F:\mathscr{E}'(V)\to \mathscr{C}^\infty(U)~\text{is continuous}.\\
        &F:W^{-s}_{\mathrm{comp}}(V)\to W^s_{\mathrm{loc}}(U)~\text{is continuous for all}~s\in\mathbb{N}_0.
    \end{split}
\end{equation}
For two continuous linear operators $A,B: C^\infty_c(V)\to\mathscr{D}'(U)$, we write $A\equiv B$ (on $U\times V$) or $A(x,y)\equiv B(x,y)$ (on $U\times V$) if $A-B$ is a smoothing operator, where $A(x,y), B(x,y)\in\mathscr{D}'(U\times V)$ are the distribution kernels of $A$ and $B$, respectively.



\subsection{Complex manifolds with boundary}\label{s-gue190320q}
Let $M$ be a relatively compact open subset with smooth boundary $X$ of a complex manifold $M'$ of dimension $n$, $n\geq3$. Let $\rho\in C^\infty(M',\Real)$ be a defining function of $X$, that is,
\[X=\{x\in M';\, \rho(x)=0\},\ \ M=\{x\in M';\, \rho(x)<0\}\]
and $d\rho(x)\neq0$ at every point $x\in X$.
Then the manifold $X$ is a CR manifold with natural CR structure $T^{1,0}X:=T^{1, 0}M'\cap \mathbb CTX$, where $T^{1,0}M'$ denotes the holomorphic tangent bundle of $M'$. Let  
$T^{0,1}M':=\overline{T^{1,0}M'}$, 
$T^{0, 1}X :=\overline{T^{1,0}X}$.



 Assume that $M'$ admits a holomorphic $d$-dimensional compact Lie group $G$ action. From now on, we will use the same assumptions and notations as in Section~\ref{s-gue190320}. Recall that we work with Assumption~\ref{a-gue2204300124}. We take a $G$-invariant Hermitian metric $\langle \cdot \,|\, \cdot \, \rangle$ on $\mathbb{C}TM'$. The $G$-invariant Hermitian metric 
 $\langle \cdot \,|\, \cdot \, \rangle$ on $\mathbb{C}TM'$ induces a 
 $G$-invariant Hermitian metric 
 $\langle \cdot \,|\, \cdot \, \rangle$ on $\mathbb{C}T^*M'$.
 
 From now on, we fix a defining function $\rho\in C^\infty(M',\Real)$ of $X$ such that
\begin{equation}\label{e-gue171006aI}
\begin{split}
&\mbox{$\langle\,d\rho(x)\,|\,d\rho(x)\,\rangle=1$ on $X$},\\
&\rho(g\circ x)=\rho(x),\ \ \forall x\in M',\ \ \forall g \in G.
\end{split}\end{equation}

Let $\frac{\partial}{\partial \rho} \in C^\infty(X, TM')$ be the global real vector field on $X$ given by
\begin{equation}
\begin{split}
&\Big\langle \frac{\partial}{\partial \rho}, d\rho \,\Big\rangle=1~\text{on}~X,\\
&\Big\langle \frac{\partial}{\partial \rho}(p)\, \big|\, v \, \Big\rangle = 0~\text{at every}~p \in X,~\text{for every}~v \in T_pX.
\end{split}
\end{equation}

Let $J : TM' \to TM'$ be the complex structure map and put
\begin{equation}
T = J\left(\frac{\partial}{\partial \rho}\, \right) \in C^\infty(X, TM'\,).
\end{equation}
The $G$-invariant Hermitian metric $\langle\,\cdot\,|\,\cdot\,\rangle$ on $\Complex TM'$ induces by duality a Hermitian metric on $\Complex T^*M'$ and Hermitian metrics on $T^{*0,q}M'$ the bundle of $(0,q)$ forms on $M'$, $q=1,\ldots,n$.
We shall also denote these Hermitian metrics by $\langle\,\cdot\,|\,\cdot\,\rangle$. Put
\[
T^{*1,0}X := (T^{0,1}X \oplus \mathbb{C}T)^\perp \subset \mathbb{C}
T^*X, \quad T^{*0,1}X := (T^{1,0}X \oplus \mathbb{C}T)^\perp \subset \mathbb{C}T^*X.
\]
Put
\begin{equation}\label{e-omega0}
\omega_0=J^t(d\rho),
\end{equation}
where $J^t$ is the complex structure map for the cotangent bundle $T^*M'$. Then, on $X$, $\omega_0 \in C^\infty(X, T^*X)$ is the global one form on $X$ satisfying
\begin{equation}
\begin{split}
& \mbox{ $\langle \omega_0(p), u \, \rangle =0$, for every $p \in X$ and every $u \in T_p^{1,0}X \oplus T_p^{0,1}X$},          \\
&  \mbox{ $\langle \omega_0, T\, \rangle =-1$ on $X$.}
\end{split}
\end{equation}

It is easy to see that under Assumption~\ref{a-gue2204300124}, the $G$-action preserves $\omega_0$. We have the pointwise orthogonal decompositions:
\begin{equation}
\begin{split}
&\mathbb{C}T^*X = T^{*1,0}X \oplus T^{*0,1}X \oplus \{ \lambda \omega_0; \lambda \in \mathbb{C} \, \}, \\
&\mathbb{C}TX =T^{1,0}X \oplus T^{0,1}X \oplus \{ \lambda T; \lambda \in \mathbb{C} \, \}.
\end{split}
\end{equation}
For $p \in X$, the Levi form of $X$ at $p$ is the Hermitian quadratic form on $T^{1,0}_pX$ given by
\begin{equation}\label{e-levi}
\mathcal{L}_p(U,\ol V)=-\frac{1}{2i}\langle\,d\omega_0(p)\,,\,U\wedge\ol V\,\rangle, \ \ \forall U, V\in T^{1,0}_pX.
\end{equation}
We can check that the Levi form on $X$ defined in  (\ref{e-levi})
is exactly
\begin{equation}\label{e-gue230327ycdm}
\mathcal{L}_p(U,\ol V)=\langle\partial\overline\partial\rho(p)\,,\,U\wedge\ol V\,\rangle,\ \ U, V\in T^{1,0}_pX.
\end{equation}

\begin{definition}\label{def-111217}
$M$ is called weakly (strongly) pseudoconvex at $x\in X$ if $\mathcal L_x$ is semi-positive (positive) definite on $T_x^{1, 0}X$. If $\mathcal L_x$ is semi-positive (positive) definite at every point of $X$, then $M$ is called a weakly (strongly) pseudoconvex  manifold.
\end{definition}

Let $A$ be a $C^\infty$ vector bundle over $M'$. 
Let $U$ be an open set in $M'$. Let 
\[
\begin{split}
&C^\infty(U\cap \ol M,A),\ \ \mathscr D'(U\cap \ol M,A),\ \ C^\infty_c(U\cap \ol M,A),\ \ 
\mathscr E'(U\cap \ol M,A),\\ 
W^s&(U\cap \ol M,A),\ \ W^s_{{\rm comp\,}}(U\cap \ol M,A),\ \ 
W^s_{{\rm loc\,}}(U\cap \ol M,A),
\end{split}
\]
(where $\ s\in\mathbb R$)
denote the spaces of restrictions to $U\cap\ol M$ of elements in 
\[
\begin{split}
C^\infty&(U\cap M',A),\ \ \mathscr D'(U\cap M',A),\ \ C^\infty(U\cap M',A),\ \ 
\mathscr E'(U\cap M',A),\\  
&W^s(M',A),\ \  W^s_{{\rm comp\,}}(M',A),\ \  
W^s_{{\rm loc\,}}(M',A),
\end{split}
\] 
respectively. Write 
\[
\begin{split}
&L^2(U\cap M,A)=L^2(U\cap\ol M,A):=W^0(U\cap \ol M,A),\\
&L^2_{{\rm comp\,}}(U\cap\ol M,A):=W^0_{{\rm comp\,}}(U\cap \ol M,A),\ \ 
L^2_{{\rm loc\,}}(U\cap\ol M,A):=W^0_{{\rm loc\,}}(U\cap \ol M,A).
\end{split}
\]
For every $q=0,\ldots,n$, we denote 
\[
\begin{split}
\Omega^{0,q}&(U\cap\ol M):=C^\infty(U\cap\ol M,T^{*0,q}M'),\ \ 
\Omega^{0,q}(M'):=C^\infty(M',T^{*0,q}M'),\\ 
&\Omega^{0,q}_c(U\cap\ol M):=C^\infty_c(U\cap\ol M,T^{*0,q}M'),\\
&\Omega^{0,q}_c(M'):=C^\infty_c(M',T^{*p,q}M'),\ \ 
\Omega^{0,q}_c(M):=C^\infty_c(M,T^{*p,q}M').
\end{split}
\]
Let $A$ and $B$ be $C^\infty$ vector bundles over $M'$. 
Let $U$ be an open set in $M'$. Let 
$$F_1, F_2: C^\infty_c(U\cap M,A)\To\mathscr D'(U\cap M,B)$$ 
be continuous operators. Let 
$F_1(x,y), F_2(x,y)\in\mathscr D'((U\times U)\cap(M\times M), A\boxtimes B^*)$ 
be the distribution kernels of $F_1$ and $F_2$ respectively. 
We write 
$$F_1\equiv F_2\!\!\mod C^\infty((U\times U)\cap(\ol M\times\ol M))$$ 
or $F_1(x,y)\equiv F_2(x,y)\!\!\mod C^\infty((U\times U)\cap(\ol M\times\ol M))$ 
if $F_1(x,y)=F_2(x,y)+r(x,y)$, where 
$r(x,y)\in C^\infty((U\times U)\cap(\ol M\times\ol M),A\boxtimes B^*)$. Similarly, 
let $\hat F_1, \hat F_2: C^\infty_c(U\cap M,A)\To\mathscr D'(U\cap X,B)$ be continuous operators. Let 
$\hat F_1(x,y), \hat F_2(x,y)\in\mathscr D'((U\times U)\cap(X\times M), A\boxtimes B^*)$ be the distribution kernels of $\hat F_1$ and $\hat F_2$ respectively. We write 
$\hat F_1\equiv\hat F_2\!\!\mod C^\infty((U\times U)\cap(X\times\ol M))$ or $\hat F_1(x,y)\equiv\hat F_2(x,y)\!\!\mod C^\infty((U\times U)\cap(X\times\ol M))$ if $\hat F_1(x,y)=\hat F_2(x,y)+\hat r(x,y)$, where 
$\hat r(x,y)\in C^\infty((U\times U)\cap(X\times\ol M),A\boxtimes B^*)$. Similarly, let $\tilde F_1, \tilde F_2: C^\infty_c(U\cap X,A)\To\mathscr D'(U\cap M,B)$ be continuous operators. Let 
\[\tilde F_1(x,y), \tilde F_2(x,y)\in\mathscr D'((U\times U)\cap(M\times X), A\boxtimes B^*)\] 
be the distribution kernels of $\tilde F_1$ and $\tilde F_2$ respectively. We write 
$\tilde F_1\equiv\tilde F_2\!\!\mod C^\infty((U\times U)\cap(\ol M\times X))$ or $\tilde F_1(x,y)\equiv\tilde F_2(x,y)\!\!\mod C^\infty((U\times U)\cap(\ol M\times X))$ if $\tilde F_1(x,y)=\tilde F_2(x,y)+\tilde r(x,y)$, where 
$\tilde r(x,y)\in C^\infty((U\times U)\cap(\ol M\times X),A\boxtimes B^*)$.

Let $(\,\cdot\,|\,\cdot\,)_{M'}$ and $(\,\cdot\,|\,\cdot\,)_{M}$ be the $L^2$ inner products on $\Omega^{0,q}_c(M')$ and $\Omega^{0,q}_c(M)$ respectively given by
\begin{equation}\label{e-gue171006aIb}
\begin{split}
&(\,u\,|\,v\,)_{M'}:=\int_{M'}\langle\,u\,|\,v\,\rangle dv_{M'},\ \ u, v\in\Omega^{0,q}_c(M'),\\
&(\,u\,|\,v\,)_{M}:=\int_M\langle\,u\,|\,v\,\rangle dv_{M'},\ \ u, v\in\Omega^{0,q}_c(M),
\end{split}
\end{equation}
where $dv_{M'}$ is the volume form on $M'$ induced by $\langle\,\cdot\,|\,\cdot\,\rangle$. Let $L^2_{(0,q)}(M)$ and $L^2_{(0,q)}(M')$ be the $L^2$ completions of $\Omega^{0,q}_c(M)$ and $\Omega^{0,q}_c(M')$ with respect to
$(\,\cdot\,|\,\cdot\,)_M$ and $(\,\cdot\,|\,\cdot\,)_{M'}$ respectively. It is clear that $\Omega^{0,q}(\ol M)\subset L^2_{(0,q)}(M)$. We write $L^2(M):=L^2_{(0,0)}(M)$. 
We extend $(\,\cdot\,|\,\cdot\,)_M$ and $(\,\cdot\,|\,\cdot\,)_{M'}$ to $L^2_{(0,q)}(M)$ 
and $L^2_{(0,q)}(M')$ in the standard way and let $\norm{\cdot}_{M}$ and $\norm{\cdot}_{M'}$ be the corresponding $L^2$ norms. Let $T^{*0,q}X$ be the bundle of $(0,q)$ forms on $X$. Recall that for every $x\in X$, we have
\[T^{*0,q}_xX:=\set{u\in T^{*0,q}_xM';\, \langle\,u\,|\,\ddbar\rho(x)\wedge g\,\rangle=0,\ \ \forall g\in T^{*0,q-1}_xM'}.\]
Let $\Omega^{0,q}(X)$ be the space of smooth $(0,q)$ forms on $X$.
Let $(\,\cdot\,|\,\cdot\,)_X$ be the $L^2$ inner product on $\Omega^{0,q}(X)$ given by
\begin{equation}\label{e-gue171006aIc}
(\,u\,|\,v\,)_X:=\int_X\langle\,u\,|\,v\,\rangle dv_X,
\end{equation}
where $dv_X$ is the volume form on $X$ induced by $\langle\,\cdot\,|\,\cdot\,\rangle$. Let $L^2_{(0,q)}(X)$ be the $L^2$ completion of $\Omega^{0,q}(X)$ with respect to
$(\,\cdot\,|\,\cdot\,)_X$. We extend $(\,\cdot\,|\,\cdot\,)_X$ to $L^2_{(0,q)}(X)$ in the standard way and let $\norm{\cdot}_X$ be the corresponding $L^2$ norm. We write $L^2(X):=L^2_{(0,0)}(X)$.

Fix $g\in G$. Let $g^*:\Lambda^r_x(\Complex T^*M')\To\Lambda^r_{g^{-1}\circ x}(\Complex T^*M')$ be the pull-back map. Since $G$ preserves $J$, we have
\[g^*:T^{*0,q}_xM'\To T^{*0,q}_{g^{-1}\circ x}M',\  \ \forall x\in M'.\]
Thus, for $u\in\Omega^{0,q}(M')$, we have $g^*u\in\Omega^{0,q}(M')$. Put
\[\Omega^{0,q}(M')^G:=\set{u\in\Omega^{0,q}(M');\, g^*u=u,\ \ \forall g\in G}.\]
Let $u\in L^2_{(0,q)}(M')$ and $g\in G$, we can also define $g^*u$ in the standard way.
Put
\[
L^2_{(0,q)}(M')^G:=\set{u\in L^2_{(0,q)}(M');\, g^*u=u,\ \ \forall g\in G}.
\]
Let $\Omega^{0,q}(\ol M)^G$ denote the space of restrictions to $M$ of elements in $\Omega^{0,q}(M')^G$. Let $L^2_{(0,q)}(M)^G$ be the completion of $\Omega^{0,q}(\ol M)^G$ with respect to $(\,\cdot\,|\,\cdot\,)_M$. Similarly, let
\begin{equation}\label{e-gue171006yc}
\Omega^{0,q}(X)^G:=\{u\in\Omega^{0,q}(X);\, g^*u=u, \ \ \forall g \in G\}.
\end{equation}
Let $L^2_{(0,q)}(X)^G$ be the completion of $\Omega^{0,q}(X)^G$ with respect to $(\,\cdot\,|\,\cdot\,)_X$. We write $L^2(X)^G:=L^2_{(0,0)}(X)^G$, $L^2(M)^G:=L^2_{(0,0)}(M)^G$, $L^2(M')^G:=L^2_{(0,0)}(M')^G$.

For $s\in\Real$, we also use $\norm{\cdot}_{s,X}$ to denote the standard Sobolev norm on $X$ of order $s$. Let $A$ be a vector bundle over $M'$. 
Let  $u\in W^s(\ol M,A)$. We define
\[\norm{u}_{s,\ol M}:=\inf\set{\norm{\Td u}_{s,M'};\, u'\in W^s(M',A), u'|_M=u}.\]
We call $\norm{u}_{s,\ol M}$ the Sobolev norm of $u$ of order $s$ on $\ol M$. 

Let $s$ be a non-negative integer. We can also define  Sobolev norm of order $s$ on $\ol M$ as follows: Let $x_0\in X$ and let $U$ be an open neighborhood of $x_0$ in $M'$ with local coordinates $x=(x_1,\ldots,x_{2n})$. 
Let $u\in\mathscr E'(U\cap\ol M)\bigcap W^s(\ol M,A)$. Let $\Td u\in \mathscr E'(U)\bigcap W^s(M',A)$ with $\Td u|_M=u$. We define the Sobolev norm of order $s$ of $u$ on $\ol M$ by
\begin{equation}\label{e-gue171007hyc}
\norm{u}^2_{(s),\ol M}:=\sum_{\alpha\in\mathbb N_0^{2n},\abs{\alpha}\leq s}\int_{M}\abs{\pr^\alpha_x\Td u}^2dv_{M'}.
\end{equation}
By using partition of unity, for $u\in W^s(\ol M,A)$, we define $\norm{u}^2_{(s),\ol M}$ in the standard way. As in function case, we define $\norm{u}^2_{(s),\ol M}$, for $u\in W^s(\ol M,A)$ in the similar way. 
It is well-known (see~\cite[Corollary B.2.6]{Hor85}) that the two norms $\norm{\cdot}_{s,\ol M}$ and $\norm{\cdot}_{(s),\ol M}$ are equivalent for every non-negative integer $s$.



\subsection{The reduction of complex manifolds with boundary}\label{s-gue2204301601}

As before, let $\mathfrak{g}$ denote the Lie algebra of $G$ and for any $\xi \in \mathfrak{g}$, we write $\xi_{M'}$ to denote the vector field on $M'$ induced by $\xi$ (see \eqref{e-gue230327ycd}). For $x \in M'$, recall that $\underline{\mathfrak{g}}_x$ is given by \eqref{e-gue2204290306}.



\begin{definition}\label{d-gue170124}
	The moment map associated to the form $\omega_0$ is the map $\mu:M' \to \mathfrak{g}^*$ defined by
	\begin{equation}\label{E:cmpm}
	\langle \mu(x), \xi \rangle = \omega_0(\xi_{M'}(x)), \qquad x \in M', \quad \xi \in \mathfrak{g}.
	\end{equation}
\end{definition}

The proof of the following lemma is standard, cf. for example, \cite[Theorem 6]{A}. 

\begin{lemma}\label{l-gue2204290819}
The moment map $\mu:M' \to \mathfrak{g}^*$ is $G$-equivariant, so $G$ acts on $Y':=\mu^{-1}(0)$, where $G$ acts on $\mathfrak{g}^*$ through co-adjoint representation.
\end{lemma}
\begin{proof}
For all $g \in G$, $\xi \in \mathfrak{g}$ and $x \in M'$, we have
\begin{eqnarray}\label{e-que04062019}
\xi_{M'}(g \circ x) & = & \frac{d}{d t}\left(\exp(t\xi)\circ g \circ x\right)|_{t=0}    \nonumber \\
&  = & \frac{d}{d t}\left(g \circ g^{-1} \circ \exp(t\xi) \circ g\circ x\right)|_{t=0}  \\
& = &  g_* \left(\operatorname{Ad} (g^{-1})\circ \xi \right)_{M'} (x) \nonumber
\end{eqnarray}
and hence
\begin{eqnarray*}
\langle \mu(g \circ x), \xi \rangle & = &\omega_0(\xi_{M'}(g \circ x))  \qquad \text{by} \  \eqref{E:cmpm}  \\
& = & \omega_0 \left( g_* \left(\operatorname{Ad} (g^{-1}) \circ \xi\right)_{M'}(x)\right) \qquad \text{by} \ \eqref{e-que04062019}  \\
& = & \omega_0 \left( \left(\operatorname{Ad}(g^{-1})\circ \xi \right)_{M'} (x) \right) \qquad \text{by $G$-invariance of $\omega$}  \\
& = & \langle \mu(x), \operatorname{Ad}(g^{-1}) \circ \xi \rangle   \qquad \text{by} \  \eqref{E:cmpm}   \\
& = &  \langle \operatorname{Ad}(g)^*\mu(x), \xi\rangle.    \nonumber
 \end{eqnarray*}
 Thus, the moment map $\mu$ is $G$-equivariant.
\end{proof}

Note that $\mu_X=\mu|_X$ is the CR moment map associated to $\omega_0$ on $X$, cf. \cite{HH, HMM}. Suppose that $\mu_X^{\, -1}(0) \not= \emptyset$, then it is shown, in \cite[Lemma 2.5]{HMM}, that if $G$ acts freely on $\mu_X^{\, -1}(0)$ and the Levi form is positive on $\mu_X^{\, -1}(0)$, then $0$ is a regular value of $\mu_X$.


Set $Y'_G:=\mu^{-1}(0)/G$. In this section, we assume that \eqref{e-gue230411yydc} holds. $\mu^{-1}(0)$ is a smooth manifold. Since $G$ acts freely on $Y'$, $Y'_G$ is a smooth manifold. Let
\[
g^{TM'} = d\omega_0(\cdot, J\cdot).
\]
Then, $g^{TM'}$ is a non-degenerate quadratic form on $TM'$ near $\mu^{-1}(0)$. 
Let $T^HY'$ be the orthogonal complement of $\underline{\mathfrak{g}}_{Y'}$ in $TY'$ with respect to $g^{TM'}$, where $\underline{\mathfrak{g}}_{Y'}:=\underline{\mathfrak{g}}|_{Y'}$. Then we have
\begin{equation}\label{e-gue2204292011}
TY' = T^HY' \oplus \underline{\mathfrak{g}}_{Y'}.
\end{equation}

\begin{lemma}\label{l-gue2204290820}
We have
\[
JT^HY' = T^HY' =JTY' \cap TY'.
\]
\end{lemma}
\begin{proof}
Since $G$ acts freely on $Y'$, then vector spaces $\underline{\mathfrak{g}}_x$ defined in \eqref{e-gue2204290306} form a vector bundle $\underline{\mathfrak{g}}$ near $\mu^{-1}(0)$. 

For $x \in Y'$, by \eqref{e-gue230411yydc} and the fact that $d\omega_0(\cdot, J\cdot\,)$ is non-degenerate on $T_xM'$, we have that $d\mu|_{TY'} =0$ and $d\mu|_{J\underline{\mathfrak{g}}_x} \to \mathfrak{g}^*$ is surjective. Since $\dim Y' + \dim \underline{\mathfrak{g}} = \dim TM'$, we have
\begin{equation}\label{e-gue230316yyd}
J\underline{\mathfrak{g}}|_{Y'} \oplus TY' = TM'|_{Y'}.
\end{equation}

 By \eqref{e-gue2204292011} and \eqref{e-gue230316yyd}, we have the $G$-equivariant orthogonal decomposition on $Y'$,
\begin{equation}\label{e-gue2204290500}
TM'|_{Y'} = \underline{\mathfrak{g}}|_{Y'}\oplus J\underline{\mathfrak{g}}|_{Y'}\oplus T^HY'.
\end{equation}
Thus from \eqref{e-gue2204290500} and $g^{TM'}$ on $TM'|_{Y'}$ is $J$-invariant, we get
\begin{equation}
JT^HY' = T^HY' =JTY' \cap TY'.
\end{equation}
\end{proof}

Let $\pi: Y' \to Y'_G$ and $\iota:Y' \hookrightarrow M'$ be the natural quotient and inclusion, respectively, then there is a unique induced 1-form $\widetilde{\omega}_0$ on $Y'_G$ such $\pi^*\widetilde{\omega}_0 = \iota^* \omega_0$.

Since $T^HY'$ is preserved by $J$, we can define the homomorphism $J_G$ on $TY'_G$ in the following way:
For $V \in TY'_G$, we denote by $V^H$ its lift in $T^HY'$, and we define $J_G$ on $Y'_G$ by
\begin{equation}\label{E:jg}
(J_GV)^H = J(V^H).
\end{equation}
Hence, we have $J_G: TY'_G \to TY'_G$ such that $J_G^2 = -\operatorname{id}$, where $\operatorname{id}$ denotes the identity map $\operatorname{id}  \, : \, TY'_G \to TY'_G.$ By complex linear extension of $J_G$ to $\mathbb{C}TY'_G$, the $\sqrt{-1}$-eigenspace of $J_G$ is given by $T^{1,0}Y'_G \, = \, \left\{ V \in \mathbb{C}TY'_G \,;\, J_GV \, =  \,  \sqrt{-1} V  \right\}.$

\begin{lemma}\label{t-gue170128b}
The almost complex structure $J_G$ is integrable, thus $(Y'_G, J_G)$ is a complex manifold.
\end{lemma}
\begin{proof}
Let $u, v \in C^\infty(Y'_G, T^{1,0}Y'_G)$, then we can find $U, V \in C^\infty(Y'_G, TY'_G)$ such that
\[
u= U - \sqrt{-1}J_GU, \qquad v=V-\sqrt{-1}J_GV.
\]
By \eqref{E:jg}, we have
\[
u^H=U^H-\sqrt{-1}JU^H, \quad v = V^H - \sqrt{-1}JV^H \in T^{1,0}X \cap \mathbb{C}TY'.
\]
Since $T^{1,0}M'$ is integrable and it is clearly that $[u^H, v^H] \in \mathbb{C}TY',$ we have $[u^H, v^H] \in T^{1,0}M' \cap \mathbb{C}TY'.$ Hence, there is a $W \in C^\infty(M', TM')$ such that
\[
[u^H, v^H] = W-\sqrt{-1}JW.
\]
In particular, $W, JW \in TY'$. Thus, $W \in TY' \cap JTY' = T^HY'$. Let $X^H \in T^HY'$ be a lift of $X \in TY'_G$ such that $X^H=W$. Then we have
\[
[u, v] = \pi_*[u^H, v^H] = \pi_*(X^H -\sqrt{-1}JX^H) = X - \sqrt{-1}J_GX \in T^{1,0}Y'_G,
\]
i.e. we have $[C^\infty(Y'_G, T^{1,0}Y'_G), C^\infty(Y'_G, T^{1,0}Y'_G)] \subset C^\infty(Y'_G, T^{1,0}Y'_G).$ Therefore, $J_G$ is integrable.
\end{proof}

Let $M'_G:=\mu^{-1}(0)/G$, $M_G:=(\mu^{-1}(0)\cap M)/G$, $X_G:=(\mu^{-1}(0)\cap X)/G$. 
By combining Lemma~\ref{l-gue2204290819} and Lemma~\ref{t-gue170128b}, we have the following 

\begin{theorem}\label{t-gue230328yyd}
Under \eqref{e-gue230411yydc}, $M'_G$ is a complex manifold of dimension $n-d$ and $M_G \subset M'_G$ is a relatively compact open subset of $M'_G$ with smooth boundary $X_G$. In particular, the Levi form of $X_G$ is negative or positive. 
\end{theorem}






\section{$G$-invariant $\overline\partial$-Neumann problem}\label{s-gue171006}

In this section, we will study $G$-invariant $\overline\partial$-Neumann problem on $M$.
Until further notice, we fix $q\in\set{0,1,\ldots,n-1}$. Let $\ddbar:\Omega^{0,q}(\ol M)\To\Omega^{0,q+1}(\ol M)$ be the Cauchy-Riemann operator.  
We extend $\ddbar$ to $L^2_{(0,q)}(M)$:
\[
\ddbar:{\rm Dom\,}\ddbar \subset L^{2}_{(0,q)}(M)\rightarrow L^{2}_{(0,q+1)}(M),
\]
where $u\in{\rm Dom\,}\ddbar$ if we can find $u_j\in\Omega^{0,q}(\ol M)$, $j=1,2,\ldots$, 
such that $u_j\To u$ in $L^2_{(0,q)}(M)$ as $j\To+\infty$ and there is a $v\in L^2_{(0,q+1)}(M)$ such that $\ddbar u_j\To v$ as $j\To+\infty$. We set $\ddbar u:=v$. 
Let
\[
\ddbar^\ast:{\rm Dom\,}\ddbar^\ast\subset L^2_{(0,q+1)}(M)\To L^2_{(0,q)}(M)
\]
be the Hilbert adjoint of $\ddbar$ with respect to $(\,\cdot\,|\,\cdot\,)_M$. The Gaffney extension of the $\ddbar$-Neumann Laplacian is given by
\begin{equation}\label{e-gue171006g}
\Box^{(q)}: {\rm Dom\,}\Box^{(q)}\subset L^2_{(0,q)}(M)\To L^2_{(0,q)}(M),
\end{equation}
where ${\rm Dom\,}\Box^{(q)}:=\{u\in L^{2}_{(0,q)}(M);\, u\in{\rm Dom\,}\overline\partial\cap{\rm Dom\,}\overline\partial^\ast, \ddbar u\in{\rm Dom\,}\ddbar^{\ast}, \ddbar^{\ast}u\in{\rm Dom\,}\ddbar \}$ and $\Box^{(q)}u=(\ddbar\,\ddbar^{\ast}+\ddbar^\ast\,\ddbar)u$, $u\in{\rm Dom\,}\Box^{(q)}$.
Put
\[
{\rm Ker\,}\Box^{(q)}=\set{u\in{\rm Dom\,}\Box^{(q)};\, \Box^{(q)}u=0}.
\]
It is easy to check that
\begin{equation}\label{e-gue171006ycb}
{\rm Ker\,}\Box^{(q)}=\{u\in{\rm Dom\,}\Box^{(q)};\, \ddbar u=0, \ddbar^{\ast} u=0\}.
\end{equation}

Since $G$ preserves $J$ and $(\,\cdot\,|\,\cdot\,)$ is $G$-invariant, it is straightforward to see that
\begin{equation}\label{e-gue161231II}
\begin{split}
&g^*\ddbar=\ddbar g^*\ \ \mbox{on ${\rm Dom\,}\ddbar$},\\
&g^*\ol{\pr}^*=\ol{\pr}^*g^*\ \ \mbox{on ${\rm Dom\,}\ol{\pr}^*$},\\
&g^*\Box^{(q)}=\Box^{(q)}g^*\ \ \mbox{on ${\rm Dom\,}\Box^{(q)}$}.
\end{split}\end{equation}
Put $({\rm Ker\,}\Box^{(q)})^G:={\rm Ker\,}\Box^{(q)}\bigcap L^2_{(0,q)}(M)^G$.

Let $\ddbar\rho^{\wedge}: T^{*0,q}M'\To T^{*0,q+1}M'$ be the operator with wedge multiplication by $\ddbar\rho$ and
let $\ddbar\rho^{\wedge,\ast}:T^{*0,q+1}M'\To T^{*0,q}M'$ be its adjoint with respect to $\langle\,\cdot\,|\,\cdot\,\rangle$, that is,
\begin{equation}\label{e-gue171006e}
\langle\,\ddbar\rho\wedge u\,|\,v\,\rangle=\langle\,u\,|\,\ddbar\rho^{\wedge,\ast} v\,\rangle,\ \ u\in T^{*0,q}M',\ \ v\in T^{*0,q+1}M'.
\end{equation}
Denote by $\gamma$ the operator of restriction on $X$. By using the calculation in page 13 of~\cite{FK72}, we can check that
\begin{equation}\label{Neumann condition}
\begin{split}
&{\rm Dom\,}\ddbar^{\ast}\cap\Omega^{0,q+1}(\ol M)=\{u\in\Omega^{0,q+1}(\ol M);\, \gamma \ddbar\rho^{\wedge,\ast}u=0 \},\\
&{\rm Dom\,}\Box^{(q)}\cap \Omega^{0,q}(\ol M)=\{u\in \Omega^{0,q}(\ol M);\,
\gamma\ddbar\rho^{\wedge,\ast}u=0,  \gamma\ddbar\rho^{\wedge,\ast}\ddbar u=0 \}.
\end{split}
\end{equation}
Let $\ddbar^\ast_f:\Omega^{0,q+1}(M')\To\Omega^{0,q}(M')$ be the formal adjoint of $\ddbar$ with respect to $(\,\cdot\,|\,\cdot\,)_{M'}$, that is,
\[(\,\ddbar u\,|\,v\,)_{M'}=(\,u\,|\,\ddbar^\ast_fv\,)_{M'},\ \ \forall u\in\Omega^{0,q}_c(M'),\ \ \forall v\in\Omega^{0,q+1}(M').\]
It is easy to see that if $u\in{\rm Dom\,}\ddbar^{\ast}\cap\Omega^{0,q+1}(\ol M)$, then
$\ddbar^{\ast}u=\ddbar^\ast_fu.$ Write $\Box^{(q)}_f=\overline\partial\,\overline\partial_f^\ast+\overline\partial_f^\ast\overline\partial$. Recall that we work with Assumption~\ref{a-gue2204300124}.

Let $\{\omega^j\}_{j=1}^n$ be an orthonormal basis of $T^{\star 1,0}M'$ in a neighborhood of $X$ with $\omega^n=\frac{\pr\rho}{\abs{\pr\rho}}$. Let $\{L_j\}_{j=1}^n$ be a dual frame of $\{\omega^j\}_{j=1}^n$ with respect to $\langle\cdot|\cdot\rangle$. It is straightforward to check that 
\begin{equation}\label{022}
T=\frac{i}{\sqrt{2}}(L_n-\bar L_n).
\end{equation}

Denote by $\bar\pr_G$ the operator $\bar\pr$ restricted on $L^2_{(0,q)}(M')^G$. As $\Box^{(q)}$, 
we can define the $G$-invariant $\bar\pr$-Laplacian:
\begin{equation}\label{e-gue230316yyda}
\Box^{(q)}_G:=\bar\pr_G^{\ast}\bar\pr_G+\bar\pr_G\bar\pr_G^{\ast}:
{\rm Dom\,}\Box^{(q)}_G\subset L^2_{(0,q)}(M)^G\rightarrow L^2_{(0,q)}(M)^G
\end{equation}
in the similar way, where $\bar\pr_G^{\ast}: {\rm Dom\,}\bar\pr_G^{\ast}\subset L^2_{(0,q+1)}(M)^G\To L^2_{(0,q)}(M)^G$ is the Hilbert space adjoint of $\bar\pr_G: {\rm Dom\,}\bar\pr_G\subset L^2_{(0,q)}(M)^G\To L^2_{(0,q+1)}(M)^G$.

\begin{lemma}\label{t-021}
Fix $q=1,\ldots,n-2$. We have 
\begin{equation}\label{e-gue230317yyd}
\|u\|_{1,\ol M}\leq C\Bigr(\|\Box_G^{(q)}u\|_{M}+\|u\|_{M}\Bigr), \forall u\in {\rm Dom}\Box_G^{(q)}\cap\Omega_G^{0, q}(\overline M),
\end{equation}
where $C>0$ is a constant. 
\end{lemma} 

\begin{proof}
Fix $p\in X$. Assume that $p\in\mu^{-1}(0)\cap X$. 
There exists a neighborhood $V$ of $p$ in $X$ such that the Levi form is positive or negative on $V$. Let $U$ be a neighborhood of $p$ in $M'$ such that $U\cap X=V$. Let 
$u\in{\rm Dom\,}\Box^{(q)}_G\cap\Omega^{0,q}(\overline M)^G$. 
Let $\chi\in C^\infty_c(U)$ and put $v:=\chi u$. Since the Levi form is positive or negative on $V$ then one has
\begin{equation}\label{024}
\|v\|^2_{1,\ol M}\leq C\Bigr(\|\overline\partial v\|^2_M+\|\overline\partial^\ast v\|^2_M+\|v\|^2_M\Bigr),
\end{equation} 
where $C>0$ is a constant independent of $u$ ($C$ depends on $\chi$). 

Now assume $p\notin \mu^{-1}(0)\cap X$. Then there exists a neighborhood $U$ of $p$ in $M'$ such that $\mu(\tilde p)\neq 0, \forall \tilde p\in U$. Moreover, we assume that $\ol U\cap \mu^{-1}(0)=\emptyset$. Let $z=(z_1,\ldots,z_n)$ be holomorphic coordinates centered at $p$ defined on an open neighborhood $U$ of $p$ in $M'$ such that \eqref{022} holds. We will use the same notations as in the discussion before \eqref{e-gue230316yyda}. On $U$, we write $u=\sideset{}{'}\sum_{|J|=q} u_J\overline \omega^J$, where $\sideset{}{'}\sum$ means that the summation is performed only over strictly increasing multiindices and for $J=(j_1,\ldots,j_q)$, $\ol\omega^J=\ol\omega^{j_1}\wedge\cdots\wedge\ol\omega^{j_q}$. Let $\chi\in C^\infty_c(U)$. 
For every strictly increasing multiindex $J$, $\abs{J}=q$, put $v_J:=\chi u_J$, 
$v=\sideset{}{'}\sum_{|J|=q} v_J\overline \omega^J$. 
We have
\begin{equation}\label{e-gue171008yb}
\begin{split}
\norm{v}^2_{1,\ol M}\leq C_1\Bigr(\sideset{}{'}\sum_{|J|=q, j\in\set{1,\ldots,n}}\|\overline L_jv_J\|^2_M+\sideset{}{'}\sum_{|J|=q, j\in\set{1,\ldots,n}}\|L_jv_J\|^2_M+\norm{v}^2_M\Bigr),
\end{split}
\end{equation}
where $C_1>0$ is a constant. Moreover, it is easy to see that
\begin{equation}\label{e1-170913}
\begin{split}
&\|\overline\partial v\|^2_M+\|\overline\partial^{\ast}v\|^2_M\\
&=\sideset{}{'}\sum_{|J|=q, j\notin J,j\in\set{1,\ldots,n}}\|\overline L_jv_J\|^2_M+
\sideset{}{'}\sum_{|J|=q, j\in J,j\in\set{1,\ldots,n}}\|L_jv_J\|^2_M+O(\|v\|_M\cdot\|v\|_{1,\ol M}).
\end{split}
\end{equation}
For $j=1,\ldots,n-1$ and every strictly increasing multiindex $J$, $\abs{J}=q$, we have
\begin{equation}\label{e-gue171008ybI}
\begin{split}
&\norm{L_jv_J}^2_M=\norm{\ol L_jv_J}^2_M+O(\norm{v}_M\norm{v}_{1,\ol M}),\\
&\norm{\ol L_jv_J}^2_M=\norm{L_jv_J}^2_M+O(\norm{v}_M\norm{v}_{1,\ol M}).
\end{split}
\end{equation}
From \eqref{e1-170913} and \eqref{e-gue171008ybI}, we deduce that
\begin{equation}\label{e1-170913r}
\begin{split}
&\|\overline\partial v\|^2_M+\|\overline\partial^{\ast}v\|^2_M\\
&=\frac{1}{2}\sideset{}{'}\sum_{|J|=q, j\in\set{1,\ldots,n-1}}\Bigr(\|\overline L_jv_J\|^2_M+\|L_jv_J\|^2_M\Bigr)\\
&+\sideset{}{'}\sum_{|J|=q, n\notin J}\|\ol L_nv_J\|^2_M+\sideset{}{'}\sum_{|J|=q, n\in J}\|L_nv_J\|^2_M+O(\|v\|_M\cdot\|v\|_{1,\ol M}).
\end{split}
\end{equation}
From \eqref{e-gue171008yb}, \eqref{e1-170913r}, we see that if $U$ is small enough, then
\begin{equation}\label{e-gue171008lv}
\|v\|^{2}_{1,\ol M}\leq C_2\Bigr(\|Tv\|^{2}_M+\|v\|^{2}_M+\|\ddbar v\|^{2}_M+\norm{\ddbar^\ast v}^2_M\Bigr),
\end{equation}
where $C_2>0$ is a constant and $\|Tv\|_M^2:=\sum_{|J|=q}^\prime\|Tv_J\|^2$.

Since $p\notin \mu^{-1}(0)\cap X$, there exists $\xi_M\in \underline{\mathfrak{g}}$ such that $\langle\omega_0,\xi_M\rangle\neq 0$ on $\overline U$ when $\overline U$ is sufficiently small. Then
\begin{equation*}
\xi_M|_X+\langle\omega_0,\xi_M\rangle T|_X\in T^{1,0}X\bigoplus T^{0,1}X.
\end{equation*}
Thus by Taylor's expansion,
\begin{equation*}
	\xi_M+\langle\omega_0,\xi_M\rangle T=\sum_{j=1}^{n-1}a_jL_j+\sum_{j=1}^{n-1}b_j\bar L_j+O(|z|)D,
\end{equation*}
where $a_j$, $b_j$ are smooth functions, $j=1,\ldots,n-1$, $D$ is a first order differential operator. Since $u\in\Omega_G^{0,q}(M)$, one has
$\xi_M v=O(\|u\|_M)$. Then
\begin{equation}\label{e-gue230317ycd}
\langle\omega_0,\xi_M\rangle Tv_J=-\xi_M v_J+\sum_{j=1}^{n-1}a_jL_jv_J+
\sum_{j=1}^{n-1}b_j\bar L_jv_J+O(|z|)Dv_J.
\end{equation}
Note that $\langle\omega_0,\xi_M\rangle\neq 0$ on $\overline U$. Then we can assume that
$|\langle\omega_0,\xi_M\rangle|\geq C>0$ on $\ol U$, where $C$ is a constant. Hence
\begin{equation}\label{e-gue230302yyd}
\|Tv\|^2_M\leq\hat C\Bigr(\sum_{j=1}^{n-1}\sideset{}{'}\sum_{|J|=q}\|\bar L_jv_J\|^2_M+\|u\|^2_M+\varepsilon_p\|v\|^2_{1,\ol M}\Big),
\end{equation}
where $\hat C>0$ is a constant and $\varepsilon_p>0$ is sufficiently small when $U$ is chosen to be small.

From \eqref{e-gue171008ybI}, \eqref{e-gue171008lv} and \eqref{e-gue230302yyd}, we deduce that 
if $U$ is small, then 
\begin{equation}\label{e-gue230302yydI}
\norm{\chi g}^2_{1,\ol M}\leq C\Bigr(\norm{\Box^{(q)}_Gg}^2_M+\norm{g}^2_M\Bigr),
\end{equation}
for all $g\in\Omega^{0,q}(\ol M)^G\cap{\rm Dom\,}\Box^{(q)}_G$, where $C>0$ is a constant independent of $g$ ($C$ depends on $\chi$).  

As before, let $u\in\Omega^{0,q}(\ol M)^G\cap{\rm Dom\,}\Box^{(q)}_G$ and let $\hat\chi\in C^\infty_c(M')$, $\hat\chi\equiv1$ near $X$, $\hat\chi\equiv0$ outside some small neighborhood 
of $X$ in $M'$. From \eqref{024}, \eqref{e-gue230302yydI} and by using partition of unity, we have 
\begin{equation}\label{e-gue230302yydII}
\norm{\hat\chi u}^2_{1,\ol M}\leq\hat C\Bigr(\norm{\Box^{(q)}_Gu}^2_M+\norm{u}^2_M\Bigr),
\end{equation}
where $\hat C>0$ is a constant independent of $u$. Since $\Box^{(q)}$ is elliptic away the 
boundary $X$, 
\begin{equation}\label{e-gue230302yydIII}
\norm{(1-\hat\chi)u}^2_{1,\ol M}\leq\tilde C\Bigr(\norm{\Box^{(q)}_Gu}^2_M+\norm{u}^2_M\Bigr),
\end{equation}
where $\tilde C>0$ is a constant independent of $u$. From \eqref{e-gue230302yydII} and 
\eqref{e-gue230302yydIII}, the lemma follows. 
\end{proof}

\begin{lemma}\label{l-02}
Fix $q=1,2,\ldots,n-2$. For all $k\in\mathbb N$, there exists $C_k>0$ such that
\begin{equation}\label{basic estimate}
\|f\|_{k,\ol M}^2\leq C_k(\|\Box_G^{(q)}f\|^2_{k-1,\ol M}+\|f\|^2_M), \forall f\in\Omega^{0,q}(\ol M)^G\cap {\rm Dom\,}\Box_G^{(q)}.
\end{equation}
\end{lemma} 

\begin{proof}
Fix $p\in X$. If $p\in\mu^{-1}(0)\cap X$, then by Assumption \ref{a-gue2204300124}, $X$ is strongly pseudoconvex or strongly pseudoconcave near $p$. Let $U$ be a neighborhood of $p$ in $M'$ such that $U\cap X$ is strongly pseudoconvex or strongly pseudoconcave and choose a cut-off function $\eta\in C_c^\infty(U)$. Then it is well-known that (see~\cite[Chapter 5]{CS01}) for $k\in\mathbb N$, 
\begin{equation}\label{e1-23-3-4}
\|\eta f\|_{k,\ol M}^2\leq C_k\Bigr(\|\Box_G^{(q)}f\|_{k-1,\ol M}^2+\|f\|^2_M\Bigr), \forall f\in\Omega^{0, q}(\overline M)^G\cap {\rm Dom\,}\Box^{(q)}_G,
\end{equation}
where $C_k>0$ is a constant independent of $f$. 

Next we assume $p\not\in\mu^{-1}(0)\cap X$. We can assume that 
there exists a neighborhood $U$ of $p$ in $M'$ such that $\overline U\cap\mu^{-1}(0)=\emptyset$ and special boundary coordinates $(t_1, t_2, \cdots, t_{2n-1}, \rho)$ centered at $p$ such that $t_1, \cdots, t_{2n-1}$ restricted to $X$ are coordinates for $X$. For $u\in C_c^\infty(U\cap\overline M)$, the partial Fourier transform of $u$ is defined by 
$$\hat u(\tau, \rho):=\int_{\mathbb R^{2n-1}}e^{-i<t, \tau>}u(t, \rho)dt,$$
where $t=(t_1,\ldots,t_{2n-1})$. 
For $s\in\mathbb R$, the tangential Sobolev norms  $|||u|||_s$ of $u$ is defined by 
$$|||u|||_s^2=\int_{\mathbb R^{2n-1}}\int_{-\infty}^0(1+|\tau|^2)^s|\hat u(\tau, \rho)|^2d\rho d\tau.$$
For $\delta>0$, $M_{\delta}$ is defined by $M_{\delta}:=\{z\in \overline M: \rho(z)>-\delta\}$. Choose a $\delta$ sufficiently small such that the tangential Sobolev norm can be defined on $M_{\delta}$ by the partition of unity and we will use $|||\cdot|||_{s(M_\delta)}$ to denote the 
tangential Sobolev norm on $M_{\delta}$. By a similar argument in the proof of \cite[Lemma 5.2.4]{CS01}, one has for every $k\in\mathbb N$, 
\begin{equation}\label{e1-23-3}
\|f\|_{k,\ol M}\leq\hat C_k\Bigr(\|\Box^{(q)}_G f\|_{k-1,\ol M}+|||f|||_{k(M_\delta)}+\|f\|_M\Bigr), \forall f\in\Omega^{0, q}(\overline M)^G\cap{\rm Dom\,}\Box^{(q)}_G,
\end{equation}
where $\hat C_k>0$ is a constant. Here, we do not need the condition that $M$ is pseudoconvex as in \cite[Lemma 5.2.4]{CS01} since we have one more term $|||f|||_{k(M_\delta)}$ in the above estimate (\ref{e1-23-3}) which can be controlled by $\|\Box^{(q)}_Gf\|_{k-1,\ol M}$ when $M$ is a bounded  strongly pseudoconvex domain.

Choose $\chi\in C_c^\infty(U\cap\overline M)$. We have for every $k\in\mathbb N$,
$$\frac{1}{\tilde C_k}\sum_{|\alpha|\leq k-1}|||D_t^\alpha(\chi f)|||_1\leq|||\chi f|||_k\leq\tilde C_k\sum_{|\alpha|\leq k-1}|||D_t^\alpha(\chi f)|||_1,$$
for every $f\in\Omega^{0,q}(\ol M)^G\cap{\rm Dom\,}\Box^{(q)}_G$, where $\tilde C_k>1$ is a constant independent of $f$.

We prove the following 

{\bf Claim:} Let $k\in\mathbb N$. For any $\varepsilon>0$, $\varepsilon\ll1$, we can take $U$ small enough so that 

\begin{equation}\label{e2-23-3-4}
|||\chi f|||_k\leq C\Bigr(\|\Box^{(q)}_Gf\|_{k-1,\ol M}+\frac{1}{\varepsilon}\|f\|_{k-1,\ol M}+\varepsilon \|\chi f\|_{k,\ol M}\Bigr), \forall f\in\Omega^{0, q}(\overline M)^G\cap{\rm Dom\,}\Box^{(q)}_G,
\end{equation}
where $C>0$ is a constant independent of $\varepsilon$. 
Let $\mathcal T^k$ denote a $k$-th order tangential differential operator of the form $D_t^\alpha$ where $|\alpha|=k$. Let $f\in\Omega^{0, q}(\overline M)^G\cap{\rm Dom\,}\Box^{(q)}_G$.
Recall that $D_{t}\chi f\in{\rm Dom\,}\overline\partial^\ast$, for any $D_t$. Then from (\ref{e-gue171008lv}) one has 
\begin{equation}\label{e10-23-3}
\begin{split}
|||\mathcal T^{k-1}(\chi f)|||_1
&\leq C\Bigr(\|T\mathcal T^{k-1}(\chi f)\|_M+\|\mathcal T^{k-1}(\chi f)\|_M+
\|\overline\partial \mathcal T^{k-1}(\chi f)\|_M+\|\overline\partial^\ast\mathcal T^{k-1}(\chi f)\|_M\Bigr)\\
&\leq C_1\Bigr( \|T\mathcal T^{k-1}(\chi f)\|_M+|||f|||_{k-1(M_\delta)}+\|\overline\partial \mathcal T^{k-1}(\chi f)\|_M+\|\overline\partial^\ast\mathcal T^{k-1}(\chi f)\|_M\Bigr),
\end{split}
\end{equation}
where $C, C_1>0$ are constants. 
It follows from \cite[(5.2.14)]{CS01} with some minor modification that
\begin{equation}\label{e11-23-3}
\begin{split}
&\|\overline\partial \mathcal T^{k-1}(\chi f)\|^2_M+\|\overline\partial^\ast\mathcal T^{k-1}(\chi f)\|^2_M\leq C_2\Bigr(|||f|||_{k-1(M_\delta)}\cdot|||\Box^{(q)}_Gf|||_{k-1(M_\delta)}+\|f\|_{k-1,\ol M}^2\\
&\quad+\norm{f}_{k-1,\ol M}|||\ddbar(\chi f)|||_{k-1}+\norm{f}_{k-1,\ol M}|||\ddbar^*(\chi f)|||_{k-1}\Bigr),
\end{split}
\end{equation}
where $C_2>0$ is a constant. 
Next, we estimate $\|T\mathcal T^{k-1}(\chi f)\|$.
From \eqref{e-gue230317ycd}, it follows that 
\begin{equation}\label{e-gue230317ycdm}
\langle \omega_0, \xi_M\rangle T\mathcal T^{k-1}(\chi f)=-\xi_M \mathcal T^{k-1}(\chi f)+\sum_{j=1}^{n-1}a_jL_j\mathcal T^{k-1}(\chi f)+\sum_{j=1}^{n-1}b_j\overline L_j\mathcal T^{k-1}(\chi f)+O(|z|)D\mathcal T^{k-1}(\chi f).
\end{equation}\label{e8-23-3}
Note that $|\langle \omega_0, \xi_M\rangle|\geq C>0$ on $\overline U$ with constant $C>0$.
Notice that 
\begin{equation}\label{e6-23-3}
\begin{split}
\xi_M\mathcal T^{k-1}(\chi f)&=[\xi_M, \mathcal T^{k-1}](\chi f)+\mathcal T^{k-1}\xi_M(\chi f)\\
&=[\xi_M, \mathcal T^{k-1}](\chi f)+\mathcal T^{k-1}[\xi_M, \chi]f+\mathcal T^{k-1}\chi\xi_M f.
\end{split}
\end{equation}
Since $f\in\Omega^{0, q}(\overline M)^G$, then one has $\xi_Mf=0$. From 
this observation and \eqref{e6-23-3}, we have 
\begin{equation}\label{e-gue230317ycdn}
\left\|-\frac{1}{\langle\omega_0, \xi_M\rangle}\xi_M\mathcal T^{k-1}(\chi f)\right\|_M\leq C_4|||f|||_{k-1(M_\delta)},
\end{equation}
where $C_4>0$ is a constant. 

Fix $j=1,\ldots,n-1$. It is straightforward to check that for every $\varepsilon>0$, we have 
\begin{equation}\label{e-gue230317ycdI}
\begin{split}
&\norm{L_j\mathcal T^{k-1}\chi f}^2_M+\norm{\ol L_j\mathcal T^{k-1}\chi f}^2_M\\
&\leq C_5\Bigr(\norm{\ddbar\mathcal T^{k-1}\chi f}^2_M+
\norm{\ddbar^*\mathcal T^{k-1}\chi f}^2_M+\frac{1}{\varepsilon}\norm{\mathcal T^{k-1}(\chi f)}^2_M
+\varepsilon\norm{f}^2_{k,\ol M}\Bigr),
\end{split}
\end{equation}
where $C_5>0$ is a constant independent of $\varepsilon$. 
From \eqref{e11-23-3} and \eqref{e-gue230317ycdI}, we deduce that 
\begin{equation}\label{e-gue230317ycda}
\begin{split}
&\norm{L_j\mathcal T^{k-1}\chi f}^2_M+\norm{\ol L_j\mathcal T^{k-1}\chi f}^2_M\\
&\leq C_6\Bigr(|||f|||_{k-1(M_\delta)}\cdot\||\Box^{(q)}_Gf\||_{k-1(M_\delta)}+\|f\|_{k-1,\ol M}^2\\
&\quad+\norm{f}_{k-1,\ol M}|||\ddbar(\chi f)|||_{k-1}+\norm{f}_{k-1,\ol M}|||\ddbar^*(\chi f)|||_{k-1}\\
&+\frac{1}{\varepsilon}\norm{\mathcal T^{k-1}(\chi f)}^2
+\varepsilon\norm{f}^2_{k,\ol M}\Bigr),
\end{split}
\end{equation}
where $C_6>0$ is a constant independent of $\varepsilon$. 
From \eqref{e-gue230317ycdm}, \eqref{e-gue230317ycdn} and \eqref{e-gue230317ycda}, we deduce that for every $\varepsilon>0$, $\varepsilon\ll1$, we can take $U$ small enough so that 
\begin{equation}\label{e-gue230318yyd} 
\norm{T\mathcal T^{k-1}(\chi f)}^2_M\leq C_7\Bigr(|||\Box^{(q)}_Gf|||^2_{k-1(M_\delta)}+\frac{1}{\varepsilon}\norm{f}^2_{k-1,\ol M}+\varepsilon\norm{\chi f}^2_{k,\ol M}\Bigr). 
\end{equation}

From \eqref{e10-23-3}, \eqref{e11-23-3} and \eqref{e-gue230318yyd}, 
we get the claim \eqref{e2-23-3-4}. From \eqref{e1-23-3-4}, \eqref{e1-23-3}, \eqref{e2-23-3-4} and by using partition of unity, the lemma follows. 
\end{proof}

From the above lemma and the techniques of elliptic regularization, one has 

\begin{theorem}\label{t-02}
 Fix $q=1,\ldots,n-2$.
 Suppose $u\in{\rm Dom\,}\Box_G^{(q)}$ and $\Box_G^{(q)}u=v$. If $v\in \Omega^{0,q}(\ol M)^G$, then $u\in \Omega^{0,q}(\ol M)^G$. If $v\in W^{s}(\ol M,T^{*0,q}M')\cap L^2_{(0,q)}(M)^G$, for some $s\in\mathbb N_0$, then $u\in W^{s+1}(\ol M,T^{*0,q}M')\cap L^2_{(0,q)}(M)^G$ and we have 
\begin{equation*}
	\|u\|_{s+1,\ol M}\leq C_s(\|v\|_{s,\ol M}+\|u\|_M), 
\end{equation*}
where $C_s>0$ is a constant independent of $u$.
\end{theorem} 

\begin{corollary}\label{c-02}
For every $q=0,\ldots,n-2$, $\Box_G^{(q)}: {\rm Dom\,}\Box^{(q)}_G\subset L^2_{(0,q)}(M)^G\To L^2_{(0,q)}(M)^G$ has closed range. In particular, for $1\leq q\leq n-2$, ${\rm Ker\,}\Box_G^{(q)}$ is a finite dimensional subspace of $\Omega^{0,q}(\ol M)^G$.
\end{corollary}

Let $N_G^{(1)}:L^2_{(0,1)}(M)^G\rightarrow {\rm Dom\,}\Box_G^{(1)}$ be the partial inverse of $\Box_G^{(1)}$. We have
\begin{equation*}
	\begin{split}
		\Box_G^{(1)}N_G^{(1)}+B_G^{(1)}&=I \ {\rm on\,} \ L^2_{(0,1)}(M)^G,\\
		N_G^{(1)}\Box_G^{(1)}+B_G^{(1)}&=I  \ {\rm on\,} \  {\rm Dom\,}\Box_G^{(1)}.
	\end{split}
\end{equation*}
By Theorem \ref{t-02}, we conclude that  $N_G^{(1)}: W^s(\overline M,T^{*0,1}M')\cap L^2_{(0,1)}(M)^G\rightarrow W^{s+1}(\overline M,T^{*0,1}M')\cap L^2_{(0,1)}(M)^G$ 
is continuous, for all $s\in\mathbb N_0$. In particular, 
\begin{equation}\label{e-gue230303yyd}
\mbox{$N^{(1)}_G: \Omega^{0,1}(\ol M)^G\To\Omega^{0,1}(\ol M)^G$ is continuous}. 
\end{equation}
Let 
\begin{equation}\label{e-gue230318yydh}
B_G: L^2(M)\To{\rm Ker\,}\Box^{(0)}_G
\end{equation}
be the orthogonal projection ($G$-invariant Bergman projection). 
Note that 
\begin{equation}\label{e-gue230319ycd}
H^0(\ol M)^G=({\rm Ker\,}\Box^{(q)})^G={\rm Ker\,}\Box^{(0)}_G. 
\end{equation}
Let 
\[Q_G: L^2(M)\To L^2(M)^G\]
be the orthogonal projection. It is not difficult to see that 
\begin{equation}\label{e-gue230303yydI}
\mbox{$Q_G: C^\infty(\ol M)^G\To C^\infty(\ol M)^G$ is continuous}. 
\end{equation} 
We can deduce from ~\cite[Theorem 4.4.5]{CS01} that 
\begin{equation}\label{e-gue230318yyda}
\mbox{$B_G=(I-\ddbar^*_GN^{(1)}_G\ddbar_G)\circ Q_G$ on $L^2(M)$}. 
\end{equation}
From \eqref{e-gue230303yyd}, \eqref{e-gue230303yydI} and \eqref{e-gue230318yyda}, we get 

\begin{theorem}\label{t-gue230318yyd}
We have $B_G: C^\infty(\ol M)\To C^\infty(\ol M)^G$ is continuous. 
\end{theorem}

\section{The Poisson operators and reduction to the boundary}\label{s-gue171010}

Until further notice, we fix $q\in\set{0,1,\ldots,n-1}$.
We first introduce some notations. We remind the reader that for $s\in\mathbb R$, the space $W^{s}(\ol M,T^{*0,q}M')$ was introduced in the discussion after Definition~\ref{def-111217}. 
Let
\[\ddbar^\ast_f: \Omega^{0,q+1}(M')\To\Omega^{0,q}(M')\]
be the formal adjoint of $\ddbar$ with respect to $(\,\cdot\,|\,\cdot\,)_{M'}$. That is,
\[(\,\ddbar f\,|\,h\,)_{M'}=(\,f\,|\,\ddbar^\ast_fh\,)_{M'},\]
$f\in\Omega^{0,q}_c(M')$, $h\in\Omega^{0,q+1}(M')$. Let
\[\Box^{(q)}_f=\ddbar\,\ddbar^\ast_f+\ddbar^\ast_f\,\ddbar: \Omega^{0,q}(M')\To\Omega^{0,q}(M')\]
denote the complex Laplace-Beltrami operator on $(0, q)$ forms. As before, let $\gamma$ denote the operator of restriction to the boundary $X$. Let us consider the map
\begin{equation}\label{e-gue171010y}
\begin{split}
F^{(q)}:W^{2}(\ol M,T^{*0,q}M')&\rightarrow L^{2}_{(0,q)}(M)\oplus
W^{\frac{3}{2}}(X,T^{*0,q}M')\\
u&\mapsto (\Box_f^{(q)}u, \gamma u).
\end{split}
\end{equation}
It is well-known that (see~\cite{B71})
$\dim{\rm Ker\,}F^{(q)}<\infty$ and ${\rm Ker\,}F^{(q)}\subset \Omega^{0,q}(\ol M)$. Let
\begin{equation}\label{e-gue171011}
K^{(q)}: W^2(\overline M, T^{*0,q}M')\rightarrow{\rm Ker\,}F^{(q)}
\end{equation}
be the orthogonal projection with respect to $(\,\cdot\,|\,\cdot\,)_M$.  Put $\tilde \Box_f^{(q)}=\Box^{(q)}_f+K^{(q)}$ and consider the map
\begin{equation}\label{e-gue171010yI}
\begin{split}
\tilde F^{(q)}: W^2(\overline M, T^{*0,q}M')&\rightarrow L^{2}_{(0,q)}(M)\oplus W^{\frac{3}{2}}(X,T^{*0,q}M')\\
u&\mapsto (\tilde\Box_f^{(q)}u, \gamma u).
\end{split}
\end{equation}
Then $\tilde F^{(q)}$ is injective (see~\cite[Part II, Chapter 3]{Hsiao08}). Let
\begin{equation}\label{e-gue171010yII}
\tilde P: C^\infty(X, T^{*0,q}M')\rightarrow\Omega^{0,q}(\overline M)
\end{equation}
be the Poisson operator for $\tilde \Box^{(q)}_f$ which is well-defined since \eqref{e-gue171010yI} is injective. The Poisson operator $\tilde P$ satisfies
\begin{equation}\label{e-gue171011II}
\begin{split}
&\tilde\Box^{(q)}_f\tilde Pu=0,\ \ \forall u\in C^\infty(X, T^{*0,q}M'),\\
&\gamma\tilde Pu=u,\ \  \forall u\in C^\infty(X, T^{*0,q}M').
\end{split}
\end{equation}

It is known that $\tilde P$ extends continuously
\[\tilde P: W^s(X, T^{*0,q}M')\rightarrow W^{s+\frac{1}{2}}(\overline M, T^{*0,q}M'),\ \ \forall s\in\mathbb R\]
(see~\cite[Page 29]{B71}). Let
$$\tilde P^\ast: \hat{\mathscr D}'(\overline M, T^{*0,q}M')\rightarrow\mathscr D'(X, T^{*0,q}M')$$ be the operator defined by
\[(\,\tilde P^\ast u\,|\,v\,)_X=(\,u\,|\,\tilde Pv\,)_M,\ \ u\in\hat{\mathscr D}'(\overline M, T^{*0,q}M'),\ \  v\in C^\infty(X, T^{*0,q}M'),\]
where $\hat{\mathscr D}'(\overline M, T^{*0,q}M')$ denotes the space of continuous linear map from $\Omega^{0,q}(\ol M)$ to $\Complex$ with respect to $(\,\cdot\,|\,\cdot\,)_M$.
It is well-known (see~\cite[page 30]{B71}) that $\tilde P^\ast$ is continuous: $\tilde P^\ast: W^s(\ol M,T^{*0,q}M')\rightarrow W^{s+\frac{1}{2}}(X,T^{*0,q}M')$ and
\[\tilde P^\ast: \Omega^{0,q}(\overline M)\rightarrow C^\infty(X, T^{*0,q}M'),\]
for every $s\in\mathbb R$.


It is well-known that the operator
\[\tilde P^\ast\tilde P:C^\infty(X,T^{*0,q}M')\To C^\infty(X,T^{*0,q}M')\]
is a classical elliptic pseudodifferential operator of order $-1$ and invertible since $\tilde P$ is
injective~(see~\cite{B71}). Moreover, the operator
\[(\tilde P^\ast\tilde P)^{-1}:C^\infty(X,T^{*0,q}M')\To C^\infty(X,T^{*0,q}M')\]
is a classical elliptic pseudodifferential operator of order $1$.

When $q=0$, we simply write $P$, $P^*$ to denote $\tilde P$, $\tilde P^\ast$ respectively.

We define a new inner product on $W^{-\frac{1}{2}}(X,T^{*0,q}M')$ as follows:
\begin{equation}\label{inner product}
[\,u\,|\,v\,]=(\,\tilde Pu\,|\,\tilde Pv)_{M},\ \ u, v\in W^{-\frac{1}{2}}(X,T^{*0,q}M').
\end{equation}
Let
\begin{equation}\label{e-gue171011pm}
Q: W^{-\frac{1}{2}}(X,T^{*0,q}M')\rightarrow{\rm Ker\,}\ddbar\rho^{\wedge,\star}
\end{equation}
be the orthogonal projection onto ${\rm Ker\,}\ddbar\rho^{\wedge,\star}$ with respect to $[\,\cdot\,|\,\cdot\,]$.

We consider the following operator
\begin{equation}\label{e-gue171011pmI}
\ddbar_{\beta}=Q\gamma\ddbar \tilde P:\Omega^{0,q}(X)\rightarrow\Omega^{0,q+1}(X).
\end{equation}
The operator $\ddbar_\beta$ was introduced by the first author in~\cite{Hsiao08}.
Let $\ddbar^\dagger_\beta:\Omega^{0,q+1}(X)\To\Omega^{0,q}(X)$ be the formal adjoint with respect to $[\,\cdot\,|\,\cdot\,]$. We recall the following (see~\cite[Part II, Lemma 5.1, equation (5.3) and Lemma 5.2]{Hsiao08})

\begin{proposition}\label{t-gue171012}
	We have that $\ddbar_\beta$ and $\ddbar^\dagger_\beta$ are classical pseudodifferential operators of order $1$,
	\begin{equation}\label{e-gue171012a}
	\ddbar_{\beta}\circ \ddbar_{\beta}=0\ \ \mbox{on $\Omega^{0,q}(X)$}
	\end{equation}
	and
	\begin{equation}\label{e-gue171012aI}
	\begin{split}
	&\mbox{$\ddbar_{\beta}=\ddbar_{b}$+lower order terms},\\
	&\mbox{$\ddbar_{\beta}^{\dagger}=\gamma\ddbar^{\star}_{f}\tilde P=\ddbar^{\star}_{b}$+lower order terms},
	\end{split}
	\end{equation}
 where $\ddbar_b$ is the tangential Cauchy Riemann operator on $X$ and $\ddbar^*_b$ is the adjoint of 
 $\ddbar_b$ with respect to $(\,\cdot\,|\,\cdot\,)_X$.
\end{proposition}

Let 
\[\Box^{(0)}_\beta:=\ddbar^\dagger_\beta\,\ddbar_\beta: \Omega^{0,q}(X)\To\Omega^{0,q}(X).\]
Let $D\subset X$ be an open set. Assume that the Levi form is positive  on $D$. By the same arguments of \cite[PartII, Chapter 7]{Hsiao08}, we can show that there are continuous operators 
$N,\hat{S}:C^\infty_c(D)\to C^\infty(D)$, which are properly supported on $D$, such that
\begin{equation}\label{e-121}
\begin{split}
\Box_\beta^{(0)}N+\hat{S}&\equiv I \ \text{on} \ D\times D, \\
\Box_\beta^{(0)}\hat{S}&\equiv 0 \ \text{on} \ D\times D.
\end{split}
\end{equation}
Moreover, $N$ is a pseudodifferential operator of order $-1$ type $(\frac{1}{2},\frac{1}{2})$ on $D$. $\hat S$ is a pseudodifferential operator of order $0$ type $(\frac{1}{2},\frac{1}{2})$ on $D$ and for any local coordinate patch $D_0\subset D$, we have 
\begin{equation}\label{e-122}
\mbox{$\hat{S}(x,y)\equiv\int_0^\infty e^{it\varphi(x,y)}s(x,y,t)dt$ on $D_0\times D_0$}, 
\end{equation}
where $\varphi\in C^\infty(D_0\times D_0)$ is the phase function as in~\cite{HM14}, $s(x,y,t)\sim \sum_{j=0}^{+\infty}s_j(x,y)t^{n-1-j}$ in $S^{n-1}_{1,0}(D_0\times D_0\times\mathbb{R}_+)$, $s_j(x,y)\in C^\infty(D_0\times D_0)$, $j=0,1,\ldots$, $s_0(x,x)=\frac{1}{2}\pi^{-n}|\det\mathcal{L}_x|$, for all $x\in D_0$, where $\det \mathcal{L}_x:=\lambda_1(x)\cdots\lambda_n(x)$, $\lambda_j(x)$, $j=1,\ldots,n$, are eigenvalues of $\mathcal{L}_x$ with respect to $\langle\,\cdot\,|\,\cdot\,\rangle$. 

Let $U$ be an open neighborhood of $\mu^{-1}(0)\cap X$ in $M'$ such that $U=GU=\{gx|g\in G, x\in U\}$. Set $D=U\cap X$, then $D=GD$.
Let $\hat{S}_G:C^\infty_c(D)\to C^\infty(D)$ be the continuous operator with distribution kernel
\begin{equation}\label{e-123}
\hat{S}_G(x, y):=\int_G \hat{S}(x,gy)d\mu(g),
\end{equation}
where $d\mu$ is the Haar measure on $G$ with $\int_Gd\mu=1$. 
Then by using the method of ~\cite{HH}, we conclude that for any local coordinate patch $D_0\subset D$, if the Levi form is negative on $D_0$, then $\hat S_G\equiv0$ on $D_0\times D_0$. If the Levi form is positive on $D_0$, then  
\begin{equation}\label{e-gue230319ycdp}
\mbox{$\hat{S}_G(x,y)\equiv\int_0^\infty e^{it\Phi(x,y)}a(x,y,t)dt$ on $D_0\times D_0$},
\end{equation}
where $\Phi\in C^\infty(D_0\times D_0)$ is the phase as in~\cite[Theorem 1.5]{HH} and 
\begin{equation}\label{e-gue230323yyd}
\begin{split}
&\mbox{$a(x,y,t)\sim\sum^{+\infty}_{j=0}a_j(x,y)t^{n-1-\frac{d}{2}-j}$ in $S^{n-1-\frac{d}{2}}_{1,0}(D_0\times D_0\times\mathbb R_+)$},\\
&\mbox{$a_j(x,y)\in C^\infty(D_0\times D_0)$, $j=0,1,\ldots$},\\
&\mbox{$a_0(x,x)$ is given by~\cite[Theorem 1.6]{HH}}.
\end{split}
\end{equation}



\section{$G$-invariant Bergman kernel asymptotics} 

We first introduce $G$-invariant Szeg\H{o} projection. 
Let $\ddbar_{b,G}:\Omega^{0,q}(X)^G\To\Omega^{0,q+1}(X)^G$ 
be the tangential Cauchy-Riemann operator on $X$ acting on $\Omega^{0,q}(X)^G$. 
We extend $\ddbar_{b,G}$ to $L^2_{(0,q)}(X)^G$: 
\begin{equation}\label{eq:2.14}
\begin{split}
&{\rm Dom\,}\ddbar_{b,G}=\set{u\in L^2_{(0,q)}(X)^G;\, 
\ddbar_{b,G}u\in L^2_{(0,q+1)}(X)^G}\,, \\
&\qquad \ddbar_{b,G}: {\rm Dom\,}\ddbar_{b,G}\ni u\longmapsto
\ddbar_{b,G}u\in L^2_{(0,q+1)}(X)^G,
\end{split}
\end{equation}
where $\ddbar_{b,G}u$ is defined in the sense of distributions. 
Let 
\[\ddbar^*_{b,G}: {\rm Dom\,}\ddbar^*_{b,G}\subset L^2_{(0,q+1)}(X)^G\To L^2_{(0,q)}(X)^G\]
be the $L^2$ adjoint of $\ddbar_{b,G}$. Let 
\begin{equation}\label{e-gue230318yydm}
\Box^{(0)}_{b,G}=\ddbar^*_{b,G}\,\ddbar_{b,G}: {\rm Dom\,}\Box^{(0)}_{b,G}\subset L^2(X)^G\To L^2(X)^G,
\end{equation}
where ${\rm Dom\,}\Box^{(0)}_{b,G}=\set{u\in L^2(X)^G;\, u\in{\rm Dom\,}\ddbar_{b,G}, \ddbar_{b,G}u\in{\rm Dom\,}\ddbar^*_{b,G}}$. Let 
\begin{equation}\label{e-gue230318yydn}
S_G: L^2(X)\To{\rm Ker\,}\Box^{(0)}_{b,G}\subset L^2(X)^G
\end{equation}
be the orthogonal projection ($G$-invariant Szeg\H{o} projection). It was proved in~\cite[Theorem 3.17]{HMM} that $\Box^{(0)}_{b,G}$ has closed range. Let 
\[N_b: L^2(X)^G\To{\rm Dom\,}\Box^{(0)}_{b,G}\]
be the partial inverse of $\Box^{(0)}_{b,G}$. We have 
\begin{equation}\label{e-gue230318yydp}
\begin{split}
\mbox{$N_b\Box^{(0)}_{b,G}+S_G=I$ on ${\rm Dom\,}\Box^{(0)}_{b,G}$},\\
\mbox{$\Box^{(0)}_{b,G}N_b+S_G=I$ on $L^2(X)^G$}.
\end{split}
\end{equation}
Recall that $B_G$ is the $G$-invariant Bergman projection (see \eqref{e-gue230318yydh}). 
We can now prove 

\begin{theorem}\label{t-027}
Let $\tau\in C^\infty(\ol M)$ with ${\rm supp\,}\tau\cap\mu^{-1}(0)\cap X=\emptyset$. 
Then, $\tau B_G\equiv0\mod C^\infty(\ol M\times\ol M)$, $B_G\tau\equiv0\mod C^\infty(\ol M\times\ol M)$.
\end{theorem}

\begin{proof}
Denote by $\gamma$ the restriction from $M'$ to $X$. It follows from the definition of $\ddbar_b$ that $\ddbar_{b,G}\gamma B_G=0$ and $\Box^{(0)}_{b,G}\gamma B_G=0$.
From this observation and \eqref{e-gue230318yydp}, we have 
\begin{equation}\label{028}
\gamma B_G=(N_b\Box^{(0)}_{b,G}+S_G)\gamma B_G=S_G \gamma B_G, 
\end{equation}
hence
\begin{equation}\label{029}
PS_G \gamma B_G=P\gamma B_G=B_G
\end{equation}
and 
\begin{equation*}
P^\ast B_G=P^\ast P\gamma B_G.
\end{equation*}
We dedeuce that 
\begin{equation}\label{0210}
\gamma B_G=(P^\ast P)^{-1}P^\ast B_G.
\end{equation}
By \eqref{029} and \eqref{0210}, we deduce that
\begin{equation}\label{0211}
B_G=PS_G(P^\ast P)^{-1}P^\ast B_G.
\end{equation}
Assume that $\tau\in C^{\infty}(\ol M), {\rm supp\,}\tau\cap\mu^{-1}(0)\cap X=\emptyset$. 
We have 
\begin{equation}\label{0212}
\tau B_G=\tau PS_G(P^\ast P)^{-1}P^\ast B_G.
\end{equation}
Let $\tilde\tau\in C^{\infty}(X)$, $\tilde{\tau}=1$ near ${\rm supp\,}\tau\cap X$ and $\tilde{\tau}=0$ on $\mu^{-1}(0)\cap X$. 
Taking adjoint in \eqref{0212}, we have
\begin{equation}\label{e-gue230306yyd}
\begin{split}
B_G\tau&=B_G P(P^\ast P)^{-1}S_GP^\ast\tau\\
&=B_G P(P^\ast P)^{-1}S_G\tilde{\tau}P^\ast\tau+B_G P(P^\ast P)^{-1}S_G(1-\tilde{\tau})P^\ast\tau.
\end{split}
\end{equation}
By \cite[Theorem 3.18]{HMM}, we know that 
\begin{equation}\label{e-gue230319yyd}
S_G\tilde{\tau}\equiv0.
\end{equation}
Moreover, from~\cite[Lemma 4.1]{HM19}, we have 
\begin{equation}\label{e-gue230319yydI}
(1-\tilde{\tau})P^\ast\tau\equiv0\mod C^\infty(X\times\ol M). 
\end{equation} 
From \eqref{e-gue230306yyd}, \eqref{e-gue230319yyd}, \eqref{e-gue230319yydI} and notice that 
$B_G: C^\infty(\ol M)\To C^\infty(\ol M)^G$ is continuous (see Theorem~\ref{t-gue230318yyd}), we deduce that 
\[B_G\tau: W^s(\ol M)\To C^\infty(\ol M),\]
for all $s\in\mathbb R$. Hence $B_G\tau\equiv0\mod C^\infty(\ol M\times\ol M)$. By taking adjoint, we deduce that $\tau B_G\equiv0\mod C^\infty(\ol M\times\ol M)$. The theorem follows. 
\end{proof}

We now study the distribution kernel of $B_G$ near $\mu^{-1}(0)\cap X$. For a Borel set $\Sigma\subset\mathbb{R}$, denote by $E(\Sigma)$ the spectral projection
of $\Box^{(0)}$ associated to $\Sigma$, where $E$ is the spectral measure of $\Box^{(0)}$. Set $H^0_{\leq\lambda}(\ol M):=Ran E\big((-\infty,\lambda]\big)$.
Let $B^{(0)}_{\leq\lambda}$ be the orthogonal projection onto $H^0_{\leq\lambda}(\ol M)$ and let $B^{(0)}(x,y)\in\mathscr D'(M\times M)$ be the distribution kernel of $B^{(0)}_{\leq\lambda}$. 
First we recall the following theorem~\cite[Theorem 1.1]{HM19}. 

\begin{theorem}\label{t-028}
Let $U$ be an open set of $M'$ with $U\cap X\neq\emptyset$. Assume that the Levi form is negative on $U\cap X$, then $B^{(0)}_{\leq\lambda}\equiv0\mod C^\infty((U\times U)\cap(\ol M\times\ol M))$. Assume that the Levi form is positive on $U\cap X$. We have
\begin{equation}\label{e-gue230319ycdm}
	B^{(0)}_{\leq\lambda}\equiv\int_0^{\infty}e^{it\phi(x,y)t}b(x,y,t)dt \ \text{mod} \ C^{\infty}\big((U\times U)\cap(\ol M\times\ol M)\big),
\end{equation}
where
\begin{equation}\label{e-gue230319ycdn}
\begin{split}
	&b(x,y,t)\sim\sum_{j=0}^\infty b_j(x,y)t^{n-j}\in S^n_{1,0}\big((U\times U)\cap(\ol M\times\ol M)\times]0,\infty[\big),\\
 &b_j(x,y)\in C^\infty((U\times U)\cap(\ol M\times\ol M)),\ \ j=0,1,\ldots,\\
 &b_0(x,y)\neq0,\\
 &\phi\in C^\infty((U\times U)\cap(\ol M\times\ol M)),\ \ {\rm Im\,}\phi\geq0,\\
 &\mbox{$\phi(x,x)=0$, $x\in U\cap X$, $\phi(x,y)\neq0$ if $(x,y)\notin{\rm diag\,}((U\times U)\cap(X\times X))$},\\
 &\mbox{${\rm Im\,}\phi(x,y)>0$ if $(x,y)\notin(U\times U)\cap(X\times X)$},\\
 &d_x\phi(x,x)=-\omega_0(x)-id\rho(x),\ \ 
 d_y\phi(x,x)=\omega_0(x)-id\rho(x),\ \ 
 \mbox{for every $x\in U\cap X$}. 
 \end{split}
\end{equation}
\end{theorem}

We refer the reader to~\cite[Theorem 5.26]{HM19} for more properties of the phase $\phi$ in \eqref{e-gue230319ycdm}. We also refer the reader to~\cite[(5.121)]{HM19} for the explicit formula for $b_0(x,x)$ in \eqref{e-gue230319ycdn}. 

From Corollary~\ref{c-02}, we know that $\Box^{(0)}_G$ has closed range. From this observation, we can repeat the proof in
~\cite[Theorem 3.17]{HMM} and deduce that there exists a constant $\lambda_0>0$ such that 
\begin{equation}\label{e-gue230319yydh}
\begin{split}
&B_G=B^{(0)}_{\leq\lambda_0}\circ Q_G\ \ \mbox{on $L^2(M)$},\\ 
&B_G(x,y)=\int_GB_{\leq\lambda_0}(x,gy)d\mu(g).
\end{split}
\end{equation}
We recall some results in~\cite{HM19}. Fix $p\in\mu^{-1}(0)$ 
and let $U$ be an open set of $p$ in $M'$ with $U=GU$. Let $D:=U\cap X$. Assume the Levi form is positive on $D$. Let 
\[L: C^\infty_c(U\cap\ol M)\To C^\infty(D)\]
be a continuous operator such that 
\begin{equation}\label{e-gue230319yydi}
L-(P^*P)^{-1}P^*\equiv0\mod C^\infty((U\times U)\cap(X\times\ol M)), 
\end{equation}
$L$ is properly supported on $U\cap\ol M$, that is, for every $\chi\in C^\infty_c(U\cap\ol M)$, there is a $\tau\in C^\infty_c(D)$ such that 
$L\chi=\tau L$ on $C^\infty_c(U\cap\ol M)$ and for every $\tau_1\in C^\infty_c(D)$, there is a $\chi_1\in C^\infty_c(U\cap\ol M)$ such that 
$\tau_1L=L\chi_1$ on $C^\infty_c(U\cap\ol M)$ (see the discussion after~\cite[Theorem 5.18]{HM19}). Since $U$ and $D$
 are $G$-invariant, we can take $L$ so that 
 \begin{equation}\label{e-gue230319yydj}
\mbox{$LQ_G=Q_{G,X}L$ on $C^\infty_c(U\cap\ol M)$},
 \end{equation}
where $Q_{G,X}$ is the orthogonal projection from $L^2(X)$ onto $L^2(X)^G$. It was shown in~\cite[(5.111), Theorem 6.11]{HM19} that
\begin{equation}\label{e-gue230319yydk}
B_{\leq\lambda_0}-P\hat SL\equiv0\mod C^\infty((U\times U)\cap(\ol M\times\ol M)),
\end{equation}
where $\hat S$ is as in \eqref{e-121}. From \eqref{e-gue230319yydh} and \eqref{e-gue230319yydk}, we deduce 
\begin{equation}\label{e-gue230319yyds}
B_G-P\hat SLQ_G\equiv0\mod C^\infty((U\times U)\cap(\ol M\times\ol M)).
\end{equation}
From \eqref{e-gue230319yydj} and \eqref{e-gue230319yyds}, we get 

\begin{theorem}\label{t-gue230319yydm}
With the notations and assumptions used above, we have 
\begin{equation}\label{e-gue230319yyydt}
B_G-P\hat S_GL\equiv0\mod C^\infty((U\times U)\cap(\ol M\times\ol M)),
\end{equation}
where $\hat S_G$ is as in \eqref{e-123}. 
\end{theorem}

By \eqref{e-gue230319ycdp} and applying the procedure in~\cite[Part II, Proposition 7.8, Theorem 7.9]{Hsiao08} we get 

\begin{theorem}\label{t-gue230319yydn}
Let $p\in\mu^{-1}(0)\cap X$. Let $U$ be an open local coordinate patch of $p$ in $M'$, $D:=U\cap X$. If the Levi form is negative on $D$, then $B_G\equiv0\mod C^\infty((U\times U)\cap(\ol M\times\ol M))$. Assume that the Levi form is positive on $D$. We have 
\begin{equation}\label{e-gue230319ycde}
B_G(z,w)\equiv\int^{+\infty}_0e^{it\Psi(z,w)}b(z,w,t)dt\mod C^\infty((U\times U)\cap(\ol M\times\ol M)),
\end{equation}
where 
\begin{equation}\label{e-gue230319ycdf}
\begin{split}
&b(z,w,t)\in S^{n-\frac{d}{2}}_{1,0}(((U\times U)\cap(\ol M\times\ol M))\times\mathbb R_+),\\
&b(z,w,t)\sim\sum^{+\infty}_{j=0}t^{n-\frac{d}{2}-j}b_j(z,w)\ \ \mbox{in $S^{n-\frac{d}{2}}_{1,0}(((U\times U)\cap(\ol M\times\ol M))\times\mathbb R_+)$},\\
&b_j(z,w)\in C^\infty((U\times U)\cap(\ol M\times\ol M)),\ \ j=0,1,2,\ldots, \\
&b_0(z,w)\neq0,
\end{split}
\end{equation}
and 
\begin{equation}\label{e-gue230319ycdg}
\begin{split}
&\Psi(z,w)\in C^\infty(((U\times U)\cap(\ol M\times\ol M))),\ \ {\rm Im\,}\Psi\geq0,\\
&\Psi(z,z)=0, \ z\in\mu^{-1}(0)\cap X,\\
&\mbox{${\rm Im\,}\Psi(z,w)>0$ if $(z,w)\notin{\rm diag\,}((\mu^{-1}(0)\cap D)\times(\mu^{-1}(0)\cap D))$},\\
&d_x\Psi(x,x)=-\omega_0(x)-id\rho(x),\ \ 
d_y\Psi(x,x)=\omega_0(x)-id\rho(x),\ \ x\in\mu^{-1}(0)\cap D,\\
&\mbox{$\Psi|_{D\times D}=\Phi$, $\Phi$ is as in \eqref{e-gue230319ycdp}}.
\end{split}
\end{equation} 

Moreover, let $z=(x_1,\ldots,x_{2n-1},\rho)$ be local coordinates of $M'$ defined near $p$ in $M'$ with $x(p)=0$ and $x=(x_1,\ldots,x_{2n-1})$ are local coordinates of $X$ defined near $p$ in $X$. Then, 
\begin{equation}\label{e-gue230319ycdaI}
\mbox{$\Psi(z,w)=\Phi(x,y)-i\rho(z)(1+f(z))-i\rho(w)(1+\ol f(w))+O(\abs{(z,w)}^3)$ near $(p,p)$},
\end{equation}
where $f\in C^\infty$, $f=O(\abs{z})$. 
\end{theorem}

From ~\cite[Part II, Proposition 7.10]{Hsiao08} we have the following

\begin{theorem}\label{t-gue230319yydp}
With the notations and assumptions used above, for $b_0(z,w)$ 
in \eqref{e-gue230319ycdf}, we have 
\[b_0(x,x)=2a_0(x,x),\ \ \mbox{for every $x\in\mu^{-1}(0)\cap D$},\]
where $a_0$ is as in \eqref{e-gue230323yyd}. 
\end{theorem} 

\begin{proof}[End of the proof of Theorem \ref{t-510}]
From Theorem~\ref{t-027}, Theorem~\ref{t-gue230319yydn} and Theorem~\ref{t-gue230319yydp}, we get Theorem~\ref{t-510}. 
\end{proof}

Let $D$ be an open set of ${\color{red}\widehat X}$ in $X$ with $D=GD$. 
Let $\chi\in C^\infty_c(D)^G$, $\chi\equiv1$ near $\mu^{-1}(0)\cap X$. Let $\chi_1\in C^\infty_c(D)^G$ with $\chi_1\equiv1$ on ${\rm supp\,}\chi$. Let $\tilde S_G:=\chi_1\hat S_G\chi: C^\infty(X)\To C^\infty(X)^G$, where $\hat S_G$ is as in \eqref{e-123}. From Theorem~\ref{t-027} and \eqref{e-gue230319yyds}, we get 

\begin{theorem}\label{t-222}
With the notations and assumptions above, we have
\[B_G\equiv P\tilde S_G(P^*P)^{-1}P^*\mod C^\infty(\overline M\times \overline M).\]
\end{theorem}


\section{The proofs of Theorem~\ref{t-gue230411yyd}, Theorem~\ref{t-que2204301824} and Theorem~\ref{t-gue230329yyd}}\label{s-gue230324yyd} 

For simplicity, until further notice, we assume that $\mu^{-1}(0)\cap X$ is strongly pseudoconvex.

For every $s\in\mathbb R$, consider the map
\begin{equation}\label{e-701}
\begin{split}
\hat{\sigma}_s=\hat{\sigma}: H^0(\overline M)^G_s&\to H^0_b(X)^G_{s-\frac{1}{2}} \\
u&\mapsto (P^\ast P)^{-1}P^\ast u=\gamma u.
\end{split}
\end{equation}
Since $(P^\ast P)^{-1} P^\ast:H^s(\overline M)^G_s\to W^{s-\frac{1}{2}}(X)$ is continuous, for every $s\in\mathbb{R}$, $\hat\sigma$ is well-defined.
We define ${\rm Coker\,}\hat\sigma_s$ in the following way:
\begin{equation}\label{e-124}
{\rm Coker\,}\hat\sigma_s={\rm Coker\,}\hat\sigma:=\{u\in H^0_b(X)^G_{s-\frac{1}{2}};\, (\,u\,|\,\hat\sigma v)_X=0, \forall v\in H^0(\overline M)^G_s\cap C^\infty(\overline M)\}.
\end{equation} 

\begin{theorem}\label{t-701}
We have that ${\rm Ker\,}\hat\sigma_s=\set{0}$
and ${\rm Coker\,}\hat\sigma_s$ is a finite dimensional subspace of $C^\infty(X)^G\cap H^0_b(X)^G$. 
Moreover, ${\rm Coker\,}\hat\sigma_s$ is independent of $s$.
\end{theorem} 

\begin{proof}
It is obvious that ${\rm Ker\,}\hat\sigma=\{0\}$. We extend $\hat\sigma$ to $\hat\sigma: \hat{\mathscr D}'(\overline M)\to\mathscr D'(X)$ by putting
\begin{equation}\label{e-gue230320yydh}
\hat\sigma u=(P^\ast P)^{-1}P^\ast B_G u=S_G(P^\ast P)^{-1}P^\ast B_G u, \ u\in\hat{\mathscr D}'(\overline M).
\end{equation}
Recall that $\hat{\mathscr D}'(\overline M)$ is the space of continuous linear maps from $C^\infty(\ol M)$ to $\mathbb C$. Since 
\[B_GP(P^*P)^{-1}S_G=(S_G(P^*P)^{-1}P^*B_G)^*\] 
maps $C^\infty(X)$ to $C^\infty(\ol M)$, the definition \eqref{e-gue230320yydh} is well-defined, where 
\[(S_G(P^*P)^{-1}P^*B_G)^*\] is the formal 
adjoint of $S_G(P^*P)^{-1}P^*B_G$ with respect to $(\,\cdot\,|\,\cdot\,)_X$ and $(\,\cdot\,|\,\cdot\,)_M$. 

Let $\hat\sigma^*: \mathscr D'(X)\to\hat{\mathscr D}'(\overline M)$ be the formal adjoint of $\hat\sigma$ with respect to the given $L^2$ inner products $(\,\cdot\,|\,\cdot\,)_M$ and $(\,\cdot\,|\,\cdot\,)_X$. We have
\begin{equation*}
\hat\sigma\hat\sigma^\ast=S_G(P^\ast P)^{-1}P^\ast B_GP(P^\ast P)^{-1}S_G.
\end{equation*}
From Theorem~\ref{t-222}, we have
\begin{equation*}
\begin{split}
\hat\sigma\hat\sigma^\ast&=S_G(P^\ast P)^{-1}P^\ast(P\tilde S_G(P^\ast P)^{-1}P^\ast+F)P(P^\ast P)^{-1}S_G\\
&=S_G\tilde S_G(P^\ast P)^{-1}S_G+S_G\hat F S_G,
\end{split}
\end{equation*}
where $F\equiv 0\mod C^\infty(\overline M\times \overline M)$ and $\hat F\equiv 0$. 
Thus
\begin{equation*}
S_G(P^\ast P)\hat\sigma\hat\sigma^\ast=S_G(P^\ast P)S_G\tilde S_G(P^\ast P)^{-1}S_G+S_G(P^\ast P)S_G\hat FS_G.
\end{equation*}
From \eqref{e-gue230319ycdp} and the complex stationary phase formula of Melin-Sj\"ostrand~\cite{MS74}, we deduce that
\begin{equation}\label{e-125}
S_G(P^\ast P)\hat\sigma\hat\sigma^\ast=S_G(I+R)S_G,
\end{equation}
where $R$ is smoothing away $\mu^{-1}(0)\cap X$ and for any $p\in\mu^{-1}(0)\cap X$, let $D$ be any small 
local coordinate patch of $X$, $p\in D$, if the Levi form is negative on $D$, then $R\equiv0$. If the Levi form is positive on $D$, we have 
\[R\equiv\int^{+\infty}_0e^{it\Psi(x,y)}r(x,y,t)dt,\]
$r(x,y,t)\in S^{n-1-\frac{d}{2}}_{1,0}(D\times D\times\mathbb R_+)$, $r(x,y,t)\sim\sum^{+\infty}_{j=0}t^{n-1-\frac{d}{2}-j}r_j(x,y)$ in $S^{n-1-\frac{d}{2}}_{1,0}(D\times D\times\mathbb R_+)$, $r_j(x,y)\in C^\infty(D\times D)$, $j=0,1,\ldots$, $r_0(x,x)=0$, for every $x\in\mu^{-1}(0)\cap D$. 
From \cite[Theorem 4.9]{HMM}, we know that ${\rm Ker\,}(I+R)$ is a finite dimensional subspace of $C^\infty(X)$. It follows from
\begin{equation*}
{\rm Coker\,}\hat\sigma_s\subset{\rm Ker\,}(I+R)\cap H^0_b(X)^G_{s-\frac{1}{2}}
\end{equation*}
that ${\rm Coker\,}\hat\sigma_s$ is a finite dimensional subspace of $C^\infty(X)^G$ and ${\rm Coker\,}\hat\sigma_s$ is independent of $s$.
\end{proof} 

As before, let $X_G:=(\mu^{-1}(0)\cap X)/G$. Let $\iota_X: \mu^{-1}(0)\cap X\To X$ be the natural inclusion and let $\iota^*_X: C^\infty(X)\To C^\infty(\mu^{-1}(0)\cap X)$ be the pull-back by $\iota_X$. Let $\iota_{X,G}: C^\infty(\mu^{-1}(0)\cap X)^G\To C^\infty(X_G)$ be the natural identification. Put 
\[\begin{split}
&H^0_b(X)^G:=\set{u\in L^2(X)^G;\, \ddbar_bu=0},\\
&H^0_b(X_G):=\set{u\in L^2(X_G);\, \ddbar_{b,X_G}u=0},
\end{split}\]
where $\ddbar_{b,X_G}$ denotes the tangential Cauchy-Riemann operator on $X_G$. For every $s\in\mathbb R$, put 
\[\begin{split}
&H^0_b(X)^G_s:=\set{u\in W^s(X);\, \ddbar_bu=0,\ \ g^*u=0,\ \ \forall g\in G},\\
&H^0_b(X_G)_s:=\set{u\in W^s(X_G);\, \ddbar_{b,X_G}u=0}.
\end{split}\] 
Let 
\[\begin{split}
\sigma_1: H^0_b(X)^G&\To H^0_b(X_G),\\
u&\To\iota_{X,G}\circ\iota^*_Xu.
\end{split}
\] 
From~\cite[Theorem 5.3]{HMM}, $\sigma_1$ extends by density to a bounded operator 
\[\sigma_1=\sigma_{1,s}: H^0_b(X)^G_s\To H^0_b(X_G)_{s-\frac{d}{4}},\]
for every $s\in\mathbb R$. Let $\triangle^X$ and $\triangle^{X_G}$ 
be the (positive) Laplacians on $X$ and $X_G$ respectively. For $s\in\mathbb R$, let $\Lambda_s:=(I+\triangle^X)^{\frac{s}{2}}$, 
$\hat\Lambda_s:=(I+\triangle^{X_G})^{\frac{s}{2}}$. Fix $s\in\mathbb R$. For $u, v\in W^s(X)$, $u', v'\in W^s(X_G)$, we define the inner products 
\[\begin{split}
&(\,u\,|\,v\,)_{X,s}:=(\Lambda_su\,|\,\Lambda_sv\,)_X,\\
&(\,u'\,|\,v'\,)_{X_G,s}:=(\hat\Lambda_su'\,|\,\hat\Lambda_sv'\,)_{X_G}
\end{split}\]
and let $\norm{\cdot}_{X,s}$ and $\norm{\cdot}_{X_G,s}$ be the corresponding norms. For every $s\in\mathbb R$, define 
\[({\rm Im\,}\sigma_{1,s})^\perp:=\set{u\in H^0_b(X_G)_{s-\frac{d}{4}};\, (\,\sigma_{1,s}v\,|\,u\,)_{X_G,s-\frac{d}{4}}=0,\ \ \forall v\in H^0_b(X)^G_s}.\]
The following theorem is the main result in \cite{HMM} (see~\cite[Theorem 1.2]{HMM})

\begin{theorem}\label{t-gue230321yyd}
With the notations and assumptions used above, assume that $\ddbar_{b,X_G}$ has closed range. Then, for every $s\in\mathbb R$, 
${\rm Ker\,}\sigma_{1,s}$ and $({\rm Im\,}\sigma_{1,s})^\perp$ are \
finite dimensional subspaces of $H^0_b(X)^G\cap C^\infty(X)^G$ and 
$H^0_b(X_G)\cap C^\infty(X_G)$ respectively, ${\rm Ker\,}\sigma_{1,s}$ and the index ${\rm dim\,}{\rm Ker\,}\sigma_{1,s}-{\rm dim\,}({\rm Im\,}\sigma_{1,s})^\perp$ are independent of $s$.
\end{theorem}

For $s\in\mathbb R$, define 
\begin{equation}\label{e-gue230321yyda}
{\rm Coker\,}\sigma_{1}={\rm Coker\,}\sigma_{1,s}:=\set{u\in H^0_b(X_G)_{s-\frac{d}{4}};\, (\,u\,|\,\sigma_{1,s}v\,)_{X_G}=0,\ \ \forall v\in H^0_b(X)^G\cap C^\infty(X)^G},
\end{equation}
where $(\,\cdot\,|\,\cdot\,)_{X_G}$ is the $L^2$ inner product on $X_G$ induced by $\langle\,\cdot\,|\,\cdot\,\rangle$.

\begin{theorem}\label{t-gue230321yydI}
With the notations and assumptions used above, assume that $\ddbar_{b,X_G}$ has closed range. Then, 
${\rm Coker\,}\sigma_{1,s}=({\rm Im\,}\sigma_{1,\frac{d}{4}})^\perp$, for every $s\in\mathbb R$. In particular, ${\rm Coker\,}\sigma_{1,s}$ is a finite dimensional subspace of $H^0_b(X_G)\cap C^\infty(X_G)$ and ${\rm Coker\,}\sigma_{1,s}$ is independent of $s$.
\end{theorem}

\begin{proof}
Let $u\in{\rm Coker\,}\sigma_{1,s}$. By definition, we have 
\begin{equation}\label{e-gue230321yydb}
(\,\sigma_{1,\frac{d}{4}}v\,|\,u\,)_{X_G}=0,\ \ \mbox{for every $v\in H^0_b(X)^G\cap C^\infty(X)^G$}. 
\end{equation}
From \eqref{e-gue230321yydb} and the proof of~\cite[Theorem 1.2]{HMM}, we can check that $u\in C^\infty(X_G)$. Thus, \eqref{e-gue230321yydb} holds for all $v\in H^0_b(X)^G_{\frac{d}{4}}$. Hence, $u\in({\rm Im\,}\sigma_{1,\frac{d}{4}})^\perp$. We have proved that ${\rm Coker\,}\sigma_{1,s}\subset({\rm Im\,}\sigma_{1,\frac{d}{4}})^\perp$. It is clear that $({\rm Im\,}\sigma_{1,\frac{d}{4}})^\perp\subset{\rm Coker\,}\sigma_{1,s}$. The theorem follows. 
\end{proof}

\begin{proof}[End of the proof of Theorem \ref{t-gue230411yyd}]
For every $s\in\mathbb R$, it is not difficult to see that $\tilde\sigma_G=\tilde\sigma_{G,s}=\hat\sigma_s\circ\sigma_{1,s}$, where $\tilde\sigma_{G,s}$ is given by \eqref{e-gue2204301555m}, \eqref{e-gue2204301610m}. From this observation, Theorem~\ref{t-701}, Theorem~\ref{t-gue230321yyd} and Theorem~\ref{t-gue230321yydI}, Theorem~\ref{t-gue230411yyd} follows. 
\end{proof} 

We now prove Theorem~\ref{t-que2204301824}. We assume that \eqref{e-gue230411yydc} holds. Recall  $\mu^{-1}(0)\cap X=\widehat X\cup\widetilde X$ on which the Levi form is strongly pseudoconvex and pseudoconcave respectively. 
As before, let $M_G=(\mu^{-1}(0)\cap M)/G$. 
For $u\in H^0_b(\widehat X_G)_s, s\in\mathbb R$, we identify $u$ with an element in $H^0_b(X_G)_s$ by putting $u=0$ on $\widetilde X_G$.
For every $s\in\mathbb R$, let 
\[\begin{split}
\sigma_{2,s}=\sigma_2: H^0_b(\widehat X_G)_s&\To H^0(\ol M_G)_{s+\frac{1}{2}},\\
u&\To B_{M_G}P_{M_G}u,
\end{split}\]
where $B_{M_G}$ and $P_{M_G}$ are the Bergman projection on $M_G$ and the Poisson operator on $M_G$ respectively. For every $s\in\mathbb R$, let 
\begin{equation}\label{e-gue230326yyd}
{\rm Coker\,}\sigma_{2}={\rm Coker\,}\sigma_{2,s}:=\set{u\in H^0(\ol M_G)_{s+\frac{1}{2}};\, (\,u\,|\,\sigma_{2,s}v\,)_{M_G}=0,\ \ \forall v\in H^0_b(\widehat X_G)\cap C^\infty(\widehat X_G)}.
\end{equation}

Now we are in a position to prove a key result. 

\begin{theorem}\label{t-gue230320ycd}
With the notations and assumptions used above, assume that $\ddbar_{b,X_G}$ has closed range.
Let $s\in\mathbb R$. We have that ${\rm Ker\,}\sigma_{2,s}$
and ${\rm Coker\,}\sigma_{2,s}$ are  finite dimensional subspaces of $H^0_b(\widehat X_G)\cap C^\infty(\widehat X_G)$ and $H^0(\ol M_G)\cap C^\infty(\ol M_G)$ respectively,  ${\rm Ker\,}\hat\sigma_{2,s}$
and ${\rm Coker\,}\hat\sigma_{2,s}$ are independent of $s$. 
\end{theorem} 

\begin{proof} 
We extend $\sigma_2$ to $\sigma_2: \mathscr D'(\widehat X_G)\To\hat{\mathscr D}'(\ol M_G)$ by putting 
\[\sigma_2u=B_{M_G}P_{M_G}S_{\widehat X_G}u,\ \ u\in\mathscr D'(\widehat X_G),\]
where $S_{\widehat X_G}$ denotes the Szeg\H{o} projection on $\widehat X_G$. Let $\sigma^*_2: \hat{\mathscr D}'(\ol M_G)\To\mathscr D'(\widehat X_G)$ 
be the formal adjoint of $\sigma_2$ with respect to $(\,\cdot\,|\,\cdot\,)_{\widehat X_G}$ and $(\,\cdot\,|\,\cdot\,)_{M_G}$. We have 
\begin{equation}\label{e-gue230321yydm}
\begin{split}
&\sigma^*_2\sigma_2=S_{\widehat X_G}P^*_{M_G}B_{M_G}P_{M_G}S_{\widehat X_G}\\
&=S_{\widehat X_G}P^*_{M_G}\Bigr(P_{M_G}\hat S_{\widehat X_G}(P^*_{M_G}P_{M_G})^{-1}P^*_{M_G}+F)P_{M_G}S_{\widehat X_G}\\
&=S_{\widehat X_G}P^*_{M_G}P_{M_G}\hat S_{\widehat X_G}S_{\widehat X_G}+S_{\widehat X_G}\hat FS_{\widehat X_G},
\end{split}
\end{equation}
where $\hat S_{X_G}$ is the operator as in \eqref{e-122}, $F\equiv0\mod C^\infty(\ol M_G\times\ol M_G)$, $\hat F\equiv0$. 
From \eqref{e-gue230321yydm}, it is straightforward to check that 
\[S_{\widehat X_G}(P^*_{M_G}P_{M_G})^{-1}\sigma^*_2\sigma_2=S_{\widehat X_G}(I+\hat R)S_{\widehat X_G},\]
where $\hat R$ is a complex Fourier integral operator of the same type, the same order, the same phase as $S_{\widehat X_G}$ but the leading term vanishes at diagonal. From this observation, we can repeat the proof of~\cite[Theorem 4.15]{HMM} with minor change and deduce that 
${\rm Ker\,}(I+\hat R)$ is a finite dimensional subspace of $C^\infty(\widehat X_G)$. Thus, ${\rm Ker\,}\hat\sigma_{2,s}$ is a finite dimensional subspace of $H^0_b(\widehat X_G)\cap C^\infty(\widehat X_G)$.

Now 
\begin{equation*}
\sigma_2(P^\ast_{M_G}P_{M_G})^{-1}\sigma_2^\ast=B_{M_G}P_{M_G}S_{\widehat X_G}(P^\ast_{M_G}P_{M_G})^{-1}S_{\widehat X_G}P^\ast_{M_G}B_{M_G}.
\end{equation*}
From \cite[Theorem 1.2]{Hsiao08}, we have
\begin{equation}\label{e-128}
\sigma_2(P^\ast_{M_G}P_{M_G})^{-1}\sigma_2^\ast=B_{M_G}(I+R_{M_G})B_{M_G},
\end{equation}
where $R_{M_G}$ is a complex Fourier integral operator of the same type, the same order, the same phase as $B_{M_G}$,  but the leading term vanishes at ${\rm diag\,}(\widehat X_G\times \widehat X_G)$.
We can deduce from \cite[Theorem 4.15]{HMM} that ${\rm Ker\,}(I+R_{M_G})$ is a finite dimensional subspace of $C^\infty(\widehat X_G)$. Hence, ${\rm Coker\,}\hat\sigma_{2,s}$ is a finite dimensional subspace of $H^0(\ol M_G)\cap C^\infty(\ol M_G)$.  The proof follows. 
\end{proof} 

\begin{proof}[End of the proof of Theorem \ref{t-que2204301824}]
It is clear that for every $s\in\mathbb R$, 
\[\sigma_G=\sigma_{G,s}=\sigma_2\circ \sigma_1\circ\hat\sigma: H^0(\overline M)_s^G\to H^0(\overline M_G)_{s-\frac{d}{4}},\]
where $\sigma_G$ is given by \eqref{e-gue2204301555}, \eqref{e-gue2204301610}. 
From Theorem~\ref{t-701}, Theorem~\ref{t-gue230321yyd}, Theorem~\ref{t-gue230321yydI} and 
Theorem \ref{t-gue230320ycd}, we deduce that $\sigma$ is Fredholm, ${\rm Ker\,}\sigma$ and ${\rm Coker\,}\sigma$ are finite dimensional subspaces of $H^0(\ol M)^G\cap C^\infty(\overline M)^G$ and $H^0(\ol M_G)\cap C^\infty(\overline M_G)$ respectively. Moreover, ${\rm Ker\,}\sigma$ and ${\rm Coker\,}\sigma$ are independent of the choices of $s$. The proof is completed.
\end{proof}

\begin{proof}[Proof of Theorem~\ref{t-gue230329yyd}]
In the end of this section, we will prove Theorem~\ref{t-gue230329yyd}. We will not assume \eqref{e-gue230411yydc}. 
Let $M_1$, $M_2$ be bounded domains in $\mathbb C^n$ with smooth boundary. Assume that $M_j$ admits a compact holomorphic Lie group $G$ action and $M_j$ satisfies Assumption~\ref{a-gue2204300124}, for each $j=1, 2$. 
 Let 
 \[\begin{split}
     &F: M_1\To M_2,\\
     &z\To (F_1(z),\ldots,F_n(z)),
 \end{split}\]
 be the $G$-invariant map which satisfies the assumption of Theorem~\ref{t-gue230329yyd}.
 We are going to prove that 
 $F$ extends smoothly to the boundary. Let $\mu_j: M_j\To\mathfrak{g}^*$ be the corresponding moment map on $M_j$ and let $X_j$ be the boundary of $M_j$, $j=1, 2$. 

 \begin{lemma}\label{l-gue230331yyd}
 Let $\tau\in C^\infty(\ol M_1)$, $\tau\equiv1$ near $\mu^{-1}_1(0)\cap X_1$. Then, $(1-\tau)F$ extends smoothly to the boundary. 
 \end{lemma}

\begin{proof}
Fix $j\in\set{1,\ldots,n}$. Since $F_{j}(z)$ is bounded, $F_{j}(z)$ is a $G$-invariant $L^2$ holomorphic function on $M_1$. We have 
\begin{equation}\label{e-gue230331yyd}
(\tau F_{j})(z)=\tau(z)(B_{G,M_1}F_{j})(z)=((\tau B_{G,M_1})F_{j})(z),
\end{equation}
where $B_{G,M_1}$ is the $G$-invariant Bergman projection on $M_1$. 
In view of Theorem~\ref{t-510}, we see that $\tau B_{G,M_1}\equiv0\mod C^\infty(\ol M_1\times\ol M_1)$. From this observation and 
\eqref{e-gue230331yyd}, we deduce that $\tau F_{j}\in C^\infty(\ol M_1)$. The lemma follows. 
\end{proof}

From Theorem~\ref{t-510}, we see that 
\begin{equation}\label{e-gue230331yydI}
B_{G,M_1}(\cdot,w)\in C^\infty(\ol M_1),\ \ \mbox{for every $w\in M_1$}. 
\end{equation}
Fix $p\in\mu^{-1}_1(0)\cap X_1$. Let $Z_1,\ldots,Z_{n-d}\in C^\infty(U,T^{1,0}\mathbb C^n)$ such that 
\begin{equation}\label{e-gue230331yydII}
{\rm span\,}\set{\eta-iJ\eta,Z_1,\ldots,Z_{n-d};\, \eta\in\underline{\mathfrak{g}}_x}=T^{1,0}_x\mathbb C^n,\ \ \mbox{for every $x\in U$},
\end{equation}
where $U$ is a small open set of $p$ in $\mathbb C^n$ and $J$ is the complex structure map on $\mathbb C^n$. It follows from \eqref{e-gue230319ycdem} and ~\cite[Theorem 1.10]{HM14} that there are $f_0,\ldots,f_{n-d}\in H^0(\ol M_1)^G\cap C^\infty(\ol M_1)$ such that 
\begin{equation}\label{e-gue230401yyd}
\begin{split}
&{\rm det\,}\left((a_{j,\ell})^{n-d}_{j,\ell=0}\right)\neq0,\\
&a_{0,\ell}=f_{\ell}(p),\ \ \ell=0,\ldots,n-d,\\
&a_{j,\ell}=(Z_{j}f_{\ell})(p),\ \ \ell=0,\ldots,n-d, j=1,\ldots,n-d.
\end{split}
\end{equation}
From \eqref{e-gue230331yydI} and \eqref{e-gue230401yyd}, combined with the proof of~\cite[Part 3 of the proof of Lemma 1]{BL80}, we deduce that there are $n-d+1$ points $a_0,a_1,\ldots,a_{n-d}$ in $M_1$ such that 
\begin{equation}\label{e-gue230401yydI}
\begin{split}
&{\rm det\,}\left((b_{j,\ell})^{n-d}_{j,\ell=0}\right)\neq0,\\
&b_{0,\ell}=B_{G,M_1}(p,a_{\ell}),\ \ \ell=0,\ldots,n-d,\\
&b_{j,\ell}=(Z_{j,x}
B_{G,M})(p,a_\ell),\ \ \ell=0,\ldots,n-d, j=1,\ldots,n-d.
\end{split}
\end{equation} 
From Lemma~\ref{l-gue230331yyd}, \eqref{e-gue230401yydI} and ~\cite[Lemma 2]{BL80}, we get Theorem~\ref{t-gue230329yyd}. 
\end{proof}

\bibliographystyle{plain}

\begin{thebibliography}{99}

\bibitem{A} C. Albert, \emph{Le th\'{e}or\`{e}me de r\'{e}duction de Marsden-Weinstein en g\'{e}om\'{e}trie cosymplectique et de contact}, J. Geom. Phys., \textbf{6} (1989), 627-649.

\bibitem{APS75} M. F. Atiyah, V. K. Patodi and I. M. Singer, \emph{Spectral asymmetry and Riemannian geometry I}, Proc. Camb. Philos. Soc. \textbf{77} (1975), 43-69.

\bibitem{BL80}
S. Bell and E. Ligocka, \emph{
A simplification and extension of Fefferman's theorem on biholomorphic mappings}, 
Invent. Math. \textbf{57} (1980), no. 3, 283--289. 

\bibitem {B71} L.~Boutet de Monvel, \emph{Boundary problems for pseudo-differential operators}, Acta Math. \textbf{126} (1971), 11-51.

\bibitem{BG81} L. Boutet de Monvel and V. Guillemin, \emph{The spectral theory of Toeplitz operators}, Ann. of Math. Stud., vol 99, Princeton Univ. Press, Princeton, NJ, 1981. v+161 pp.

\bibitem{BouSj76}
L.~Boutet~de Monvel and J.~Sj{\"o}strand, \emph{Sur la singularit{\'e} des
noyaux de {B}ergman et de {S}zeg{\"o}}, Ast{\'e}risque, \textbf{34--35} (1976), 123--164.

\bibitem{Ch06}
L.~Charles,
\emph{Toeplitz operators and Hamiltonian torus actions},
J. Funct. Anal. \textbf{236} (2006), no. 1, 299--350. 

\bibitem{CS01}
S.-C.~Chen and M.-C.~Shaw, \emph{Partial differential equations in several complex variables},
AMS/IP Studies in Advanced Mathematics. 19. Providence, RI: American Mathematical Society (AMS).
Somerville, MA: International Press, xii, 380 p., (2001). 

\bibitem{Fer74} 
C.~Fefferman, \emph{The {Bergman} kernel and biholomorphic 
mappings of pseudoconvex domains}, Invent. Math. \textbf{26} (1974), 1--65.

\bibitem{FK72} G. B.~Folland and J. J.~Kohn, \emph{The {Neumann} problem for the {Cauchy}-{Riemann} complex}, Annals of Mathematics Studies, No. 75. Princeton University Press, Princeton, N.J.; University of Tokyo Press, Tokyo, 1972. viii+146 pp.



\bibitem{GS} V. Guillemin and S. Sternberg, \emph{Geometric quantization and multiplicities of group representations}, Invent. Math. \textbf{67} (1982), no. 3, 515-538.

\bibitem{Hor85} L.~H\"ormander, \emph{The analysis of linear partial differential operators}. III. Pseudodifferential operators. Grundlehren der Mathematischen Wissenschaften [Fundamental Principles of Mathematical Sciences], 274. Springer-Verlag, Berlin, 1985. viii+525 pp.


\bibitem{HS17} P. Hochs and Y. Song, \emph{Equivariant indices of Spin-Dirac operators for proper moment maps}, Duke Math. J. \textbf{166} (2017), 1125-1178.



\bibitem{Hsiao08} C.-Y.~Hsiao, \emph{Projections in several complex variables}, M\'em. Soc. Math. France, Nouv. S\'er. \textbf{123} (2010). 


\bibitem{HHLS20} C.-Y.~Hsiao, R.-T.~Huang, X.~Li and G.~Shao, \emph{$S^1$-equivariant index theorems and Morse inequalities on complex manifolds with boundary}, J. Funct. Anal. \textbf{279} (2020), no. 3, 51 pp. 

\bibitem{HH} C.-Y. Hsiao and R.-T. Huang, \emph{$G$-invariant Szeg\H{o} kernel asymptotics and CR reduction}, Calculus of Variations and PDEs. \textbf{60} (2021), no 1, paper No. 47.


\bibitem{HM14} C.-Y.~Hsiao and G.~Marinescu, \emph{On the singularities of the Szeg\H{o} projections on lower energy forms}, J. Differential Geom. \textbf{107} (2017), no. 1, 83--155.

\bibitem{HM19} C.-Y. Hsaio and G Marinescu, \emph{On the singularities of the Bergman projections for lower energy forms on complex manifolds with boundary}, arXiv: 1911.10928. 



\bibitem{HMM} C.-Y. Hsiao, X. Ma and G. Marinescu, \emph{Geometric quantization on CR manifolds}, arXiv: 1906.05627, to appear at Communications in Contemporary Mathematics. DOI: 10.1142/s0219199722500742

\bibitem{HS22}
C.-Y.~Hsiao and N.~Savale, 
\emph{Bergman-Szeg\H{o} kernel asymptotics in weakly pseudoconvex 
finite type cases},  J. Reine Angew. Math. \textbf{791} (2022), 173--223.

 \bibitem{M10}  X.~Ma,
\emph{Geometric quantization on K{\"a}hler and symplectic manifolds},
International {C}ongress of {M}athematicians,
vol.\ II, Hyderabad, India, August 19-27 (2010), 785--810.


 \bibitem{MZ09} X. Ma and W. Zhang,\emph{Geometric quantization for proper moment maps}, C. R. Math. Acad. Sci. Paris, \textbf{347},  (2009), 389-394.


 \bibitem{MZ14} X. Ma and W. Zhang,\emph{Geometric quantization for proper moment maps: the Vergne conjecture}, Acta Math., \textbf{212},  (2014), no. 1, 11--57.


\bibitem{M96} E. Meinrenken, \emph{On Riemann-Roch formulas for multiplicities}, J. Amer. Math. Soc. \textbf{9} (1996), 373-389.


\bibitem{M98} E. Meinrenken, \emph{Symplectic surgery and the Spinc-Dirac operator},  Adv. Math. \textbf{134} (1998), no. 2, 240-277.

\bibitem{MS74}
A.~Melin and J.~Sj\"{o}strand,
\emph{Fourier integral operators with complex-valued phase functions},
Springer Lecture Notes in Math., \textbf{459}, (1975), 120--223. 

\bibitem{MZ} X. Ma and W. Zhang, 
\emph{Bergman kernels and symplectic reduction}, 
Ast\'{e}risque \textbf{318} (2008), viii+154 pp. 

\bibitem{Pa03} R. Paoletti, \emph{Moment maps and equivariant Szeg\H{o} kernels}, J. Symplectic Geom. 2 (2003), \textbf{1}, 133--175.

\bibitem{P11} P.-E. Paradan, \emph{Formal geometric quantization II}, Pacific J. Math. \textbf{253} (2011), 169-211.

\bibitem{TZ98} Y. Tian and W. Zhang, \emph{An analytic proof of the geometric quantization conjecture of Guillemin-Sternberg}, Invent. Math. \textbf{132}, (1998), no. 2, 229-259.

\bibitem{TZ99} Y. Tian and W. Zhang, \emph{Quantization formula for symplectic manifolds with boundary}, Geom. Funct. Anal. \textbf{9}, (1999), no. 3, 596-640.


\bibitem{V96} M. Vergne, \emph{Multiplicities formula for geometric quantization I, II}, Duke Math. J. \textbf{82} (1996), 143-179, 181-194.

\bibitem{V02} M. Vergne, \emph{Quantification g\'{e}om\'{e}trique et r\'{e}duction symplectique}, S\'{e}minaire Bourbaki, Vol. 2000/2001. Asr\'{e}risque No. \textbf{282} (2002), Exp. No. 888, viii, 249-278.

\bibitem{V07} M. Vergne, \emph{Applications of equivariant cohomology}, International Congress of Mathematicians. Vol. I, Eur. Math. Soc., Z\"{u}rich, 2007, pp. 635-664.




\end{thebibliography}

\end{document}